\documentclass{amsart}

\usepackage{amsfonts}
\usepackage{amssymb}
\usepackage{tikz}
\usetikzlibrary{arrows,decorations.pathreplacing}
\usetikzlibrary{shapes,snakes}
\usetikzlibrary{arrows,calc} 

\allowdisplaybreaks

\newcommand{\tikzAngleOfLine}{\tikz@AngleOfLine}
\def\tikz@AngleOfLine(#1)(#2)#3{%
\pgfmathanglebetweenpoints{%
\pgfpointanchor{#1}{center}}{%
\pgfpointanchor{#2}{center}}
\pgfmathsetmacro{#3}{\pgfmathresult}%
}

\usepackage[left=25mm,right=25mm,top=25mm,bottom=25mm]{geometry}
\usepackage{graphicx}
\usepackage{bm}
\usepackage[
	pdftitle={},
	pdfauthor={},
	ocgcolorlinks,
	linkcolor=linkblue,
	citecolor=linkred,
	urlcolor=linkred]
{hyperref}
\usepackage{color}
\definecolor{linkred}{rgb}{0.75,0,0}
\definecolor{linkblue}{rgb}{0,0,0.75}
\usepackage{mathpazo}
\usepackage{microtype}
\usepackage{multicol}
\usepackage{booktabs}
\usepackage{latexsym}
\usepackage{multirow}
\usepackage{setspace}
\usepackage{enumerate}

\theoremstyle{plain}
\newtheorem{theorem}{Theorem}
\newtheorem{proposition}{Proposition}[section]
\newtheorem{lemma}[proposition]{Lemma}

\newcommand{\bt}{\begin{theorem}}
\newcommand{\et}{\end{theorem}}

\theoremstyle{definition}
\newtheorem{definition}[proposition]{Definition}
\newtheorem{remark}[proposition]{Remark}
\newcommand{\beq}{\begin{equation}}
\newcommand{\eeq}{\end{equation}}
\newcommand{\bl}{\begin{lemma}}
\newcommand{\el}{\end{lemma}}

\newcommand{\cal}{\mathcal}
\newcommand{\Res}{\mathop{\,\rm Res\,}}

\newcommand{\ca}{\mathcal{A}}

\newcommand{\ce}{\mathcal{E}}

\newcommand{\cp}{{\cal P}}
\newcommand{\bc}{\mathbb{C}}
\newcommand{\bn}{\mathbb{N}}
\newcommand{\bp}{\mathbb{P}}

\newcommand{\bz}{\mathbb{Z}}
\newcommand {\h}{\hbar}

\newcommand{\modm}{\cal M}
\newcommand{\un}{1\!\!1}

\begin{document}
	
\title{Gromov-Witten invariants of $\bp^1$ coupled to a KdV tau function.}
\author{Paul Norbury}
\address{School of Mathematics and Statistics, University of Melbourne, VIC 3010, Australia}
\email{\href{mailto:norbury@unimelb.edu.au}{norbury@unimelb.edu.au}}
\thanks{}
\date{\today}

\begin{abstract}
We consider the pull-back of a natural sequence of cohomology classes $\Theta_{g,n}\in H^{2(2g-2+n)}(\overline{\modm}_{g,n})$ to the moduli space of stable maps $\overline{\modm}_{g,n}(\bp^1,d)$.  These classes are related to the Br\'ezin-Gross-Witten tau function of the KdV hierarchy via
$Z^{BGW}(\h,t_0,t_1,...)=\exp\sum\frac{\h^{2g-2}}{n!}\int_{\overline{\modm}_{g,n}}\Theta_{g,n}\cdot\prod_{j=1}^n\psi_j^{k_j}\prod t_{k_j}$.  Insertions of the pull-backs of the classes $\Theta_{g,n}$ into the integrals defining Gromov-Witten invariants define new invariants which we show in the case of target $\bp^1$ are given by a random matrix integral and satisfy the Toda equation.
\end{abstract}

\maketitle

\tableofcontents

\section{Introduction}

Let $X$ be a projective algebraic variety and consider $(C,x_1,\dots,x_n)$ a connected smooth curve of genus $g$ with $n$ distinct marked points.  For ${\bf d} \in H_2(X,\bz)$ the moduli space of stable maps $\overline{\modm}_{g,n}(X,{\bf d})$ is defined by:
$$\overline{\modm}_{g,n}(X,{\bf d})=\{(C,x_1,\dots,x_n)\stackrel{f}{\rightarrow} X\mid f_\ast [C]={\bf d}\}/\sim$$
where $f$ is a morphism from a connected nodal curve $C$ containing distinct points $\{x_1,\dots,x_n\}$ that avoid the nodes.  Any genus zero irreducible component of $C$ with fewer than three distinguished points (nodal or marked), or genus one irreducible component of $C$ with no distinguished point, must not be collapsed to a point under $f$.  The quotient is by isomorphisms of the domain $C$ that fix each $x_i$ and commute with $f$.    The moduli space of stable maps has irreducible components of different dimensions but it has a virtual class which is a well-defined cycle in the Chow group of $\overline{\modm}_{g,n}(X,{\bf d})$ of dimension
$$ \dim[\overline{\modm}_{g,n}(X,{\bf d})]^{\text{vir}}=(\dim X-3)(1-g)+\langle c_1(X),{\bf d}\rangle +n.
$$  
When $X=\{\text{pt}\}$, $\overline{\modm}_{g,n}(X,{\bf d})=\overline{\modm}_{g,n}$ is the moduli space %(Deligne-Mumford stack) 
of genus $g$ stable curves---curves with only nodal singularities and finite automorphism group---with $n$ labeled points disjoint from nodes.  For $2g-2+n>0$ there is a forgetful map $p:\overline{\modm}_{g,n}(X,{\bf d})\to \overline{\modm}_{g,n}$, which maps a stable map to its domain curve followed by contraction of unstable components. 

In \cite{NorNew}, a collection of natural cohomology classes 
$$\Theta_{g,n}\in H^{2(2g-2+n)}(\overline{\cal M}_{g,n})$$ 
associated to the Br\'ezin-Gross-Witten tau function of KdV %for $g\geq 0$, $n\geq 0$ and $2g-2+n>0$ 
is defined---see Section~\ref{sec:theta}.   Consider the pull-back class $\Theta_{g,n}^X=p^*\Theta_{g,n}\in H^*(\overline{\modm}_{g,n}(X,{\bf d}))$ with respect to the forgetful map $p$.  We will treat $\Theta_{g,n}^X$ interchangeably as a cohomology class and a Chow class since $\Theta_{g,n}$ was constructed as a Chow class in \cite{NorNew}. This class cuts out a smaller virtual fundamental class of dimension 
\begin{equation} \label{thetdim}
\dim \left\{[\overline{\modm}_{g,n}(X,{\bf d})]^{\text{vir}}\cap \Theta_{g,n}^X\right\}=(\dim X-1)(1-g)+\langle c_1(X),{\bf d}\rangle.
\end{equation} 
Define the $\Theta$-Gromov-Witten invariants of a variety $X$ by coupling to the classes $\Theta_{g,n}$ as follows. For $2g-2+n>0$ define
$$ %\begin{equation} \label{gwtheta}
\left\langle \Theta\cdot\prod_{i=1}^n\tau_{b_i}(\alpha_i) \right\rangle ^{g}_{{\bf d}}:=\int_{[{\modm}_{g,n}(X,{\bf d})]^{vir}} \Theta^X_{g,n}\cdot\prod_{i=1}^n\psi_i^{b_i}ev_i^\ast(\alpha_i)
%\left\langle \Theta\cdot\prod_{i=1}^l\tau_{b_i}(1)  \prod_{i=l+1}^n\tau_{b_i}(\omega)\right\rangle ^{g}_{d}:=\int_{[\overline{\modm}^g_n(\bp^1,d)]^{vir}} \Theta^X_{g,n}\cdot\prod_{i=1}^l\psi_i^{b_i} \prod_{i=l+1}^n\psi_i^{b_i}ev_i^\ast(\omega)
$$
where the classes $\alpha_i\in H^*(X)$ are pulled back by the evaluation maps
$ev_i:\overline{\modm}_{g,n}(X,{\bf d})\to X$ and $\psi_i=c_1(L_i)$ for the line bundle $L_i\to\overline{\modm}_{g,n}(X,{\bf d})$ with fibre the cotangent line $T^*_{x_i}C$ of the $i$th labeled point on the domain curve.   

When $2g-2+n\leq 0$, the $\Theta$-Gromov-Witten invariants are defined as follows.  For $(g,n)=(0,2)$ define the genus 0 two-point $\Theta$-Gromov-Witten invariants to coincide with the usual genus 0 two-point Gromov-Witten invariants of $X$.  For $(g,n)=(0,1)$ the genus 0 one-point $\Theta$-Gromov-Witten invariants are obtained from the genus 0 two-point invariants via the forgetful map---see Section~\ref{GW}.  For $(g,n)=(1,0)$, in the Fano case, which we need here, we define the invariants to vanish.    One property of the classes $\Theta_{g,n}$ is the vanishing of the genus zero classes $\Theta_{0,n}=0$.  Hence the genus 0 $\Theta$-Gromov-Witten invariants vanish for $n\geq 3$.

In this paper we mainly consider the target $X=\bp^1$.  In this case the dimension formula \eqref{thetdim} is independent of $g$ and $n$ and reduces to 
$$\dim \left\{[\overline{\modm}_{g,n}(\bp^1,d)]^{\text{vir}}\cap \Theta_{g,n}^{\bp^1}\right\}=2d$$
for $d\in\bn$ the degree of the map.  Let $1,\omega\in H^*(\bp^1)$ be the identity and the K\"ahler class.   Assemble the $\Theta$-Gromov-Witten invariants of $\bp^1$ into the free energy defined by
\begin{align}  \label{freenergy}
F_{\bp^1}^{\Theta}(\h,&s_0,s_1,...,t_0,t_1,...)=\sum_{g\geq 0} \h^{2g-2} F_g^{\Theta}=  \sum_{g,d}\h^{2g-2}\left\langle\Theta\cdot\exp\left\{\sum_{i\geq 0}^{\infty}\tau_i(\omega)s_i+\tau_i(1)t_i\right\}\right\rangle^g_d+\frac14\log{\h}\\
=&-\h^{-2}s_1(1-t_0)+\h^{-2}s_0^2+\frac14\log{\frac{\h}{1-t_0}}+\frac14\frac{s_1}{1-t_0}-\frac{\h^2}{64(1-t_0)^2}+...   \nonumber
\end{align}
and define the partition function $Z^{\Theta}_{\bp^1}(\h,\{s_k,t_k\})=\exp F^{\Theta}_{\bp^1}(\h,\{s_k,t_k\})$.  
The $\frac14\log{\h}$ term is included so that $Z_{\bp^1}(\h,\{s_k,t_k\})$ is homogeneous---see \eqref{freenerg} in Section~\ref{sec:cohft}.

The following theorem expresses the stationary invariants in terms of a random matrix integral.
\begin{theorem}  \label{partint}
Set $\h=N^{-1}$ and $t_k=0$ for $k\geq 0$ in the partition function.   Then
\begin{equation}   \label{partleg}
Z_{\bp^1}^{\Theta}(\h=N^{-1},\{s_k\},\{t_k=0\})=\frac{c}{N!}\int_{-2}^2\int_{-2}^2...\int_{-2}^2dx_1...dx_N\prod_{i<j}(x_i-x_j)^2\exp\left(N{\sum_{k\geq 0} s_k\sum_{i=1}^N x_i^{k+1}}\right)
\end{equation}
where $c\in\bc$ is a constant.
More precisely, $Z_{\bp^1}^{\Theta}$ coincides with an asymptotic expansion of the integral as $N\to\infty$.
\end{theorem}
The 0-point $\Theta$-Gromov-Witten invariants, which can be calculated by setting $\{s_k=0\}$ in \eqref{partleg}, are non-trivial for all $g\neq 1$.  For the usual Gromov-Witten invariants the only non-trivial 0-point invariant is in genus 0, due to the degree one self-map of $\bp^1$.  Whereas, conveniently 
$$\deg\Theta_g=2g-2=\dim[\overline{\modm}_{g}(\bp^1,0)]^{\text{vir}}$$
so they pair to give non-trivial invariants.  %(In fact, the $g=1$ 0-point invariant vanishes.)

Theorem~\ref{partint} yields new information about the classes $\Theta_{g,n}\in H^*(\overline{\modm}_{g,n})$.  In Section~\ref{0-point} it is shown that \eqref{partleg} produces the following integrals involving Hodge classes
$$\int_{\overline{\modm}_{g,1}}\lambda_{g-1}\Theta_{g,1}=\frac{(-1)^{g-1}(1-2^{-2g})}{g}B_{2g},\quad g\geq 1
$$
where $B_k$ is the $k$th Bernoulli number.  

The integral \eqref{partleg} also gives relations among the usual Gromov-Witten invariants of $\bp^1$.  For example, $\Theta_2=\frac18\delta_{00}+\frac38\delta_{01}$ where $\delta_{00}$, respectively $\delta_{01}$, is the stratum of $\overline{\modm}_{2}$ consisting of irreducible, respectively reducible, curves with two nodes.   Hence
$$-\frac{1}{64}=\int_{[\overline{\modm}_{2}(\bp^1,0)]^{\text{vir}}}\Theta_2=\frac{1}{16}\langle\tau_0(1)^2\tau_0(\omega)^2\rangle^0+\frac38\langle\tau_0(1)\rangle^1\langle\tau_0(1)\tau_0(\omega)^2\rangle^0+\frac38\langle\tau_0(\omega)\rangle^1\langle\tau_0(1)^2\tau_0(\omega)\rangle^0
$$
which uses formula  \eqref{div} for Gromov-Witten invariants integrated over lower strata.  We can explicitly check the right hand side above since $\langle\tau_0(\omega)\rangle^1=-\frac{1}{24}$, $\langle\tau_0(1)^2\tau_0(\omega)\rangle^0=1$ and $\langle\tau_0(1)^2\tau_0(\omega)^2\rangle^0=0=\langle\tau_0(1)\rangle^1$. 

%$$\int_{\overline{\modm}_{2}}\Theta_{2}\cdot\Omega_2=\frac{9}{2}\langle\overline{\tau}_2(1)\overline{\tau}_2(1)\rangle-15\langle\overline{\tau}_3(1)\rangle$$

Previously it was not known that the random matrix integral \eqref{partleg} in Theorem~\ref{partint} is related to intersection theory on $\overline{\modm}_{g,n}(\bp^1,d)$.  Theorem~\ref{partint} can be viewed in two ways.  It brings a geometric meaning to the integral \eqref{partleg}, and used in reverse it enables the calculation of the $\Theta$-Gromov-Witten invariants of $\bp^1$.  One immediate application is a proof that the partition function $Z_{\bp^1}^{\Theta}$ of $\Theta$-Gromov-Witten invariants of $\bp^1$ satisfies the Toda equation, which was already known for the partition function $Z_{\bp^1}$ of the usual Gromov-Witten invariants of $\bp^1$---defined in Section~\ref{sec:toda}.   

To study the Toda equation we include insertions of the class $\tau_0(1)$ i.e. consider $Z_{\bp^1}^{\Theta}(\h,\{s_k,t_k=0,k>0\})$ and substitute $s_0=\h(1-t_0)\tilde{s}_0$ into $Z_{\bp^1}^{\Theta}(\h,\{s_k\},t_0)$. %$\frac{\partial}{\partial \tilde{s}_0}$ means $t_0$ and $\{s_k,k>0\}$ are held constant.
\begin{theorem}   \label{Todath}
The partition function satisfies the Toda equation
$$\frac{\partial^2}{\partial \tilde{s}_0^2}\log Z_{\bp^1}^{\Theta}(t_0)=\frac{Z_{\bp^1}^{\Theta}(t_0+\h)Z_{\bp^1}^{\Theta}(t_0-\h)}{Z_{\bp^1}^{\Theta}(t_0)^2}
$$
where $Z_{\bp^1}^{\Theta}(t_0):=Z_{\bp^1}^{\Theta}(\h,\{s_k\},t_0)|_{s_0=\h(1-t_0)\tilde{s}_0}$.
\end{theorem}

A specialisation of the partition function gives rise to the following wave function which satisfies a differential equation known as the quantum curve.  Define the wave function
\begin{equation}   \label{wave}
\psi(x,\h)=Z_{\bp^1}^{\Theta}\Big(\{s_k=\h\frac{k!}{x^{k+1}},t_k=0\}\Big).
%=\exp\left\{\frac{q}{\h^2}+\frac{-\frac{\h}{24}+\frac{q}{\h}+\frac{1}{2}t^2\h}{x}+\frac{..}{x^2}+...\right\}
\end{equation}
\begin{theorem}   \label{th:wave}
The wave function \eqref{wave} satisfies the quantum curve equation   
$$\left((4-x^2)\h^2\frac{d^2}{dx^2}-2x\h^2\frac{d}{dx}+1+\h\right)\psi(x,\h)=0.
$$
\end{theorem}
The quantum curve is essentially Legendre's differential equation.  The expression $x^{-1/\h}\psi(x,\h)$ is analytic at $x=\infty$ with coefficients of its Taylor series in $x^{-1}$ rational in $\h$.  Its first few terms are given by
$$x^{-1/\h}\psi(x,\h)=1+x^{-2}(-\hbar^{-1}+\frac12+\frac14\h+\frac18\h^2+...)
$$
The coefficient of $x^{-2d}\hbar^{-\chi}$ in $x^{-1/\h}\psi(x,\h)$ collects all of the stationary invariants arising from degree $d$ maps to $\bp^1$ with disconnected domain of Euler characteristic $\chi$.  The same specialisation of the partition  function for usual Gromov-Witten invariants $Z_{\bp^1}\Big(\{s_k=\h\frac{k!}{x^{k+1}},t_k=0\}\Big)$ is also the wave function for a quantum curve \cite{DMNPSQua}, which is a difference equation instead of a differential equation.
 
The wave function has a WKB expansion $\log\psi(x,\h)=\sum_{k\geq 0}\h^{k-1}S_k(x)$ and the differential equation gives $(4-x^2)\left(\frac{dS_0}{dx}\right)^2+1=O(\h)$ hence the differential equation in Theorem~\ref{th:wave} is supported on a curve.  This is known as the $\h\to 0$ semi-classical limit of the quantum curve which produces the spectral curve
\begin{equation}   \label{specurve} 
S^{\Theta}_{\bp^1}=\{(x,y)\in\bc^2\mid (x^2-4)y^2=1\}. %x=z+\frac{q}{z},\quad y=\frac{z}{z^2-q}
\end{equation}
For $k>0$, the $S_k(x)$ are expansions at $x=\infty$ of analytic functions defined on the spectral curve.

Theorem~\ref{partint} produces only the stationary invariants $Z_{\bp^1}^{\Theta}(\h=N^{-1},\{s_k\},\{t_k=0\})$.   The full partition function, allowing $t_k\neq 0$, i.e insertions of non-stationary terms can be obtained from Theorem~\ref{TRGW} below which is essentially a replacement for Virasoro constraints.   The usual Gromov-Witten invariants of $\bp^1$ satisfy Virasoro constraints.  This is proven in \cite{OPaVir} where the Virasoro constraints are presented as decay conditions that allow one to calculate non-stationary invariants from stationary invariants.  The Virasoro constraints have an alternative formulation in terms of topological recursion which is a procedure that takes as input a spectral curve and produces a collection of correlators defined on $C^n$.
\begin{theorem}  \label{TRGW}
Topological recursion applied to the spectral curve $S^{\Theta}_{\bp^1}$ can be used to produce the partition function
$$Z_{\bp^1}^{\Theta}(\h=N^{-1},\{s_k\},\{t_k\}).$$
More precisely, the $\Theta$-Gromov-Witten invariants are obtained via contour integrals of the correlators $\omega_{g,n}^{S^{\Theta}_{\bp^1}}$. 
\end{theorem}
Theorem~\ref{TRGW} produces stationary $\Theta$-Gromov-Witten invariants via 
$$\left\langle \Theta\cdot\prod_{i=1}^n \tau_{k_i}(\omega) \right\rangle^g_d=\Res_{z_1=0}...\Res_{z_n=0}\prod_{i=1}^n\frac{x_i^{k_i+1}}{(k_i+1)!}\omega_{g,n}^{S^{\Theta}_{\bp^1}}.$$
Section~\ref{proofs} describes how non-stationary $\Theta$-Gromov-Witten invariants are obtained from the correlators $\omega_{g,n}^{S^{\Theta}_{\bp^1}}$. 

Theorem~\ref{TRGW} is proven by pushing forward cohomology classes on $\overline{\modm}_{g,n}(\bp^1,d)$ to cohomology classes on $\overline{\modm}_{g,n}$ and calculating so-called {\em ancestor} invariants.  This is the content of Section~\ref{sec:cohft}.  

The paper is organised as follows. In Section~\ref{sec:theta} we define the cohomology classes $\Theta_{g,n}\in H^*(\overline{\modm}_{g,n})$ and describe their relation to the Br\'ezin-Gross-Witten tau function of KdV.  In Section~\ref{sec:cohft} we define cohomological field theories (CohFTs) which provide the crucial link between the left and right hand sides of \eqref{partleg} in Theorem~\ref{partint}---Gromov-Witten invariants can be calculated using an associated CohFT; and a matrix integral satisfies loop equations defined on a spectral curve which is associated to a CohFT via topological recursion.  In Section~\ref{GW} we prove properties of $\Theta$-Gromov-Witten invariants for a general target variety $X$.  Section~\ref{Giv} describes a graphical formulation of Givental's action on CohFTs and the resulting construction of semisimple CohFTs.  Section~\ref{sec:TR} describes an equivalent formulation of this graphical construction via topological recursion.  The main theorems are proven in Section~\ref{proofs}.   A number of explicit calculations can be found in Sections~\ref{1-point} and \ref{0-point}.

\vspace{.5cm}
\noindent{\em Acknowledgements.}  I would like to thank Peter Forrester for useful conversations and the referee for many useful comments. This research was supported by the Australian Research Council grants DP170102028 and DP180103891.

\section{Cohomology classes over $\overline{\modm}_{g,n}$ and KdV tau functions}   \label{sec:theta}

Let $\overline{\modm}_{g,n}$ be the moduli space %(Deligne-Mumford stack) 
of genus $g$ stable curves---curves with only nodal singularities and finite automorphism group---with $n$ labeled points disjoint from nodes.   Define $\psi_i=c_1(L_i)\in H^{2}(\overline{\mathcal{M}}_{g,n},\mathbb{Q})$, the first Chern class of the line bundle $L_i\to\overline{\mathcal{M}}_{g,n}$ with fibre above $[(C,p_1,\ldots,p_n)]$ given by $T_{p_i}^*C$.  

The KdV equation
\begin{equation}\label{kdv}
U_{t_1}=UU_{t_0}+\frac{\h^2}{12}U_{t_0t_0t_0},\quad U(t_0,0,0,...)=f(t_0)
\end{equation}
is the first equation in a hierarchy of equations $U_{t_k}=P_k(U,U_{t_0},U_{t_0t_0},...)$ for $k>1$ which determine $U$ uniquely from $U(t_0,0,0,...)$.  This is known as the KdV hierarchy---see \cite{MJDSol} for the full definition.  A tau function $Z(\h,t_0,t_1,...)$ of the KdV hierarchy %(equivalently the KP hierarchy in odd times $p_{2k+1}=t_k/(2k+1)!!$) 
gives rise to a solution $U=\h^2\frac{\partial^2}{\partial t_0^2}\log Z$ of the KdV hierarchy.

Witten conjectured, and Kontsevich proved, that a generating function for the intersection numbers $\int_{\overline{\cal M}_{g,n}}\prod_{i=1}^n\psi_i^{m_i}$ is a tau function of the KdV hierarchy.
\begin{theorem}[Witten-Kontsevich 1992 \cite{KonInt,WitTwo}] \label{th:KW}
\begin{equation}  \label{KW}
Z^{\text{KW}}(\h,t_0,t_1,...)=\exp\sum_{g,n}\h^{2g-2}\frac{1}{n!}\sum_{\vec{k}\in\bn^n}\int_{\overline{\modm}_{g,n}}\prod_{i=1}^n\psi_i^{k_i}t_{k_i}
\end{equation}
is a tau function of the KdV hierarchy.%$$ F^{\text{KW}}(\h,t_0,t_1,...)=\sum_g\h^{2g}\langle\tau_0^{k_0}\tau_1^{k_1}...\rangle_g\frac{t_0^{k_0}}{k_0!}\frac{t_1^{k_1}}{k_1!}...$$
\end{theorem}

The Kontsevich-Witten tau function $Z^{\text{KW}}(\h,t_0,t_1,...)$ is defined by the initial condition $U^{\text{KW}}(\h,t_0,0,...)=t_0$ for $U^{\text{KW}}=\h^2\frac{\partial^2}{\partial t_0^2}\log Z^{\text{KW}}$.  The low genus terms of $\log Z^{\text{KW}}$ are 
$$\log Z^{\text{KW}}(\h,t_0,t_1,...)=\h^{-2}(\frac{t_0^3}{3!}+\frac{t_0^3t_1}{3!}+\frac{t_0^4t_2}{4!}+...)+\frac{t_1}{24}+...
$$

In \cite{NorNew} we constructed cohomology classes $\Theta_{g,n}\in H^*(\overline{\modm}_{g,n})$ for $g\geq 0$, $n\geq 0$ and $2g-2+n>0$ such that the generating function $Z^{\text{BGW}}(\h,t_0,t_1,...)$ of intersection numbers of $\Theta_{g,n}$ with $\psi$ classes is also a tau function of the KdV hierarchy.   The function $Z^{\text{BGW}}(\h,t_0,t_1,...)$ is known as the Br\'ezin-Gross-Witten tau function and arose previously in the study of $U(n)$ matrix models \cite{BGrExt,GWiPos}.  It is defined by the initial condition $U^{\text{BGW}}(\h,t_0,0,0,...)=\frac18\h^2/(1-t_0)^2$ for $U^{\text{BGW}}=\h^2\frac{\partial^2}{\partial t_0^2}\log Z^{\text{BGW}}$.  The low genus $g$ terms of $\log Z^{\text{BGW}}$ are
\begin{align*}  %\label{lowg}
\log Z^{\text{BGW}}&=-\frac{1}{8}\log(1-t_0)+\h^2\frac{3}{128}\frac{t_1}{(1-t_0)^3}+\h^4\frac{15}{1024}\frac{t_2}{(1-t_0)^5}+\h^4\frac{63}{1024}\frac{t_1^2}{(1-t_0)^6}+O(\h^6)\\
&=(\frac{1}{8}t_0+\frac{1}{16}t_0^2+\frac{1}{24}t_0^3+...)+\h^2(\frac{3}{128}t_1+\frac{9}{128}t_0t_1+...)+\h^4(\frac{15}{1024}t_2+\frac{63}{1024}t_1^2+...)+...\nonumber
\end{align*}
\begin{theorem}[\cite{NorNew}]  \label{th:theta}
There exist cohomology classes $\Theta_{g,n}\in H^{2(2g-2+n)}(\overline{\modm}_{g,n})$ so that
$$Z^{BGW}(\h,t_0,t_1,...)=\exp\sum\frac{\h^{2g-2}}{n!}\int_{\overline{\modm}_{g,n}}\hspace{-3mm}\Theta_{g,n}\cdot\prod_{j=1}^n\psi_j^{k_j}\prod t_{k_j}%=-\frac{1}{8}\log(1-t_0)+\h\frac{3}{128}\frac{t_1}{(1-t_0)^3}+O(\h^2).
$$
is the Br\'ezin-Gross-Witten tau function of the KdV hierarchy.
\end{theorem}

The classes $\Theta_{g,n}\in H^{2(2g-2+n)}(\overline{\modm}_{g,n})$ are constructed in \cite{NorNew} via the push-forward of classes defined on the moduli space of stable spin curves.  %For any stable graph $\Gamma$ of genus $g$ and with $n$ external edges define $$\phi_{\Gamma}:\overline{\modm}_{\Gamma}=\prod_{v\in V(\Gamma)}\overline{\modm}_{g(v),n(v)}\to\overline{\modm}_{g,n},\quad \Theta_{\Gamma}=\prod_{v\in V(\Gamma)}\pi_v^*\Theta_{g(v),n(v)}\in H^*(\overline{\modm}_{\Gamma})$$ where $\pi_v$ is projection onto the factor $\overline{\modm}_{g(v),n(v)}$. 
They naturally restrict to the boundary divisors %in $\overline{\modm}_{g,n}$, i.e. the image of the gluing maps 
$\overline{\modm}_{g-1,n+2}\to\overline{\modm}_{g,n}$  and $\overline{\modm}_{h,n_1+1}\times\overline{\modm}_{g-h,n_2+1}\to\overline{\modm}_{g,n}$ and are compatible with the forgetful map $\overline{\modm}_{g,n+1}\stackrel{\pi}{\longrightarrow}\overline{\modm}_{g,n}$.

\begin{enumerate}[(i)]
\item  $\Theta_{g,n}|_{\overline{\modm}_{g-1,n+2}}=\Theta_{g-1,n+2}$,\qquad $\Theta_{g,n}|_{\overline{\modm}_{g_1,n_1+1}\times\overline{\modm}_{g_2,n_2+1}}=\Theta_{g_1,n_1}\cdot\Theta_{g_2,n_2}$  \label{cohft}  
\item $\Theta_{g,n+1}=\psi_{n+1}\cdot\pi^*\Theta_{g,n}$,  \label{forget}
\item  $\Theta_{1,1}=3\psi_1$  \label{base}
\end{enumerate}
Properties \eqref{cohft}-\eqref{base} are enough to uniquely determine all intersection numbers $\int_{\overline{\modm}_{g,n}}\Theta_{g,n}\cdot\prod_{j=1}^n\psi_j^{k_j}$ hence the partition function in Theorem~\ref{th:theta}.   Property~\eqref{cohft} shows that $\Theta_{g,n}$ is a {\em degenerate cohomological field theory} discussed in Section~\ref{sec:cohft}.  Property~\eqref{forget} implies that 
$$\pi_*(\Theta_{g,n+1}\psi_{n+1}^m)=\Theta_{g,n}\kappa_m.$$ 
This push-forward property has a graphical interpretation which is crucial in the sequel.  It allows one to calculate Gromov-Witten invariants coupled to pull-backs of the classes $\Theta_{g,n}$ using the same graphical data---described in Section~\ref{Giv}---to calculate usual Gromov-Witten invariants.

\section{Cohomological field theories}  \label{sec:cohft}

Cohomology classes on the moduli space of stable maps $\overline{\modm}_{g,n}(X,{\bf d})$ intersected with the virtual fundamental class $[\overline{\modm}_{g,n}(X,{\bf d})]^{\text{vir}}$ naturally push forward to cohomology classes on the moduli space of stable curves $\overline{\modm}_{g,n}$.  The Gromov-Witten invariants of a target variety $X$ can be calculated via these push-forward classes on $\overline{\modm}_{g,n}$ using the beautiful structure of a cohomological field theory (CohFT) which is defined below.  Importantly for this paper, the pull-back of the class $\Theta_{g,n}$ to the moduli space of stable maps is realised simply by cup product of $\Theta_{g,n}$ with the push-forward of classes from $\overline{\modm}_{g,n}(X,{\bf d})$.    In this section we describe the Gromov-Witten invariants of a target variety $X$ coupled to the classes $\Theta_{g,n}$ from this perspective.  We begin more generally, with a description of the cup product of $\Theta_{g,n}$ with any CohFT.

A {\em cohomological field theory} is a pair $(V,\eta)$ consisting of a finite-dimensional complex vector space $V$ equipped with a non-degenerate symmetric bilinear form $\eta$, referred to as a metric, and a sequence of $S_n$-equivariant maps: 
\[ \Omega_{g,n}:V^{\otimes n}\to H^*(\overline{\modm}_{g,n})\]
that satisfy compatibility conditions from inclusion of strata:
%\[ \overline{\modm}_{g_1,n_1+1}\times\overline{\modm}_{g_2,n_2+1}\to\overline{\modm}_{g_1+g_2,n_1+n_2}\]
%$\displaystyle I_{g,n}(\alpha_1\otimes\cdots\otimes\alpha_n)=I_{g_1,n_1+1}\otimes I_{g_2,n_2+1}\big(\bigotimes_{j\in S_1}\alpha_j\otimes\Delta\otimes \bigotimes_{j\in S_2}\big)$ where $\Delta\in V\otimes V$ is dual to $\eta$.  
$$\phi_{\text{irr}}:\overline{\modm}_{g-1,n+2}\to\overline{\modm}_{g,n},\quad \phi_{h,I}:\overline{\modm}_{h,|I|+1}\times\overline{\modm}_{g-h,|J|+1}\to\overline{\modm}_{g,n},\quad I\sqcup J=\{1,...,n\}$$
%then for $\displaystyle\Delta=\sum_{\alpha,\beta}\eta^{\alpha\beta}e_{\alpha}\otimes e_{\beta}$
given by
\begin{equation}\label{glue}
\begin{split}
\phi_{\text{irr}}^*\Omega_{g,n}(v_1\otimes...\otimes v_n)&=\Omega_{g-1,n+2}(v_1\otimes...\otimes v_n\otimes\Delta) \\
\phi_{h,I}^*\Omega_{g,n}(v_1\otimes...\otimes v_n)&=\Omega_{h,|I|+1}\otimes \Omega_{g-h,|J|+1}\big(\bigotimes_{i\in I}v_i\otimes\Delta\otimes\bigotimes_{j\in J}v_j\big)
\end{split}
\end{equation}
where $\Delta\in V\otimes V$ is dual to the metric $\eta\in V^*\otimes V^*$, and there exists a vector $\un\in V$ satisfying 
\begin{equation} \label{unmet}
\Omega_{0,3}(v_1\otimes v_2\otimes \un)=\eta(v_1,v_2).
\end{equation}
which is essentially a non-degeneracy condition.
When $n=0$, $\Omega_g:=\Omega_{g,0}\in H^*(\overline{\modm}_{g})$.   A CohFT defines a product $\cdot$ on $V$ using the non-degeneracy of $\eta$ by
\begin{equation}  \label{prod} 
\eta(a\cdot b,c)=\Omega_{0,3}(a,b,c).
\end{equation}
and $\un$ is a unit for the product.    We will also consider sequences of $S_n$-equivariant maps $\Omega_{g,n}$ that satisfy \eqref{glue}, but do not satisfy \eqref{unmet} which we call {\em degenerate} CohFTs.
 %In local coordinates it is given by $\Delta=\eta^{\alpha\beta}e_{\alpha}\otimes e_{\beta}$.  

The partition function of a (degenerate) CohFT $\Omega=\{\Omega_{g,n}\}$ with respect to a basis $\{e_1,...,e_D\}$ of $V$, is defined by:
\begin{equation}   \label{partfun}
Z_{\Omega}(\h,\{\bar{t}^{\alpha}_k\})=\exp\sum_{g,n,\vec{\alpha},\vec{k}}\frac{\h^{2g-2}}{n!}\int_{\overline{\modm}_{g,n}}\Omega_{g,n}(e_{\alpha_1}\otimes...\otimes e_{\alpha_n})\cdot\prod_{j=1}^n\psi_j^{k_j}\prod\bar{t}^{\alpha_j}_{k_j}
\end{equation}
 $\alpha_i\in\{1,...,D\}$ and $k_j\in\bn$.  

\begin{definition} \label{cohfttop}
Given a CohFT $\Omega$, define $\Omega^{\text{top}}$ to be its degree 0 part:
$$\Omega^{\text{top}}_{g,n}(v_1\otimes...\otimes v_n)=\deg_0\Omega_{g,n}(v_1\otimes...\otimes v_n)\in H^0(\overline{\modm}_{g,n}). 
$$
\end{definition}
Note that $\Omega^{\text{top}}$ is a CohFT since it is $S_n$-equivariant and conditions \eqref{glue} and \eqref{unmet} restrict to $\Omega^{\text{top}}_{g,n}$.  The CohFT $\Omega^{\text{top}}$ is also known as a (two-dimensional) {\em topological} field theory, since together with the metric $\eta$ it satisfies axioms of Atiyah \cite{AtiTop} on a functor from cobordisms in $1+1$ dimensions to vector spaces.
 
When $\dim V=1$, we identify $\Omega_{g,n}$ with the image $\Omega_{g,n}(\un^{\otimes n})$ for a choice of basis vector $\un\in V$ and we write $\Omega_{g,n}\in H^*(\overline{\modm}_{g,n})$.  The trivial CohFT is $\Omega_{g,n}=1\in H^0(\overline{\modm}_{g,n})$, which is also a topological field theory.  The classes $\Omega_{g,n}=\Theta_{g,n}\in H^0(\overline{\modm}_{g,n})$ form a degenerate CohFT by the restriction property \eqref{cohft} in Section~\ref{sec:theta}.  By Theorem~\ref{th:KW}, the partition function of the trivial CohFT $\Omega_{g,n}=1\in H^*(\overline{\modm}_{g,n})$ is $Z_{\Omega}(\h,\{t_k\})=Z^{\text{KW}}(\h,\{t_k\})$, and by Theorem~\ref{th:theta}, the partition function of the degenerate CohFT $\Omega_{g,n}=\Theta_{g,n}\in H^*(\overline{\modm}_{g,n})$ is $Z_{\Omega}(\h,\{t_k\})=Z^{\text{BGW}}(\h,\{t_k\})$.

A CohFT may satisfy the further condition that insertion of the unit vector is compatible with the forgetful map $\pi:\overline{\modm}_{g,n+1}\to\overline{\modm}_{g,n}$ in the following way.  A CohFT that satisfies 
\begin{equation}  \label{flatunit} 
\Omega_{g,n+1}(v_1\otimes...\otimes v_n\otimes \un)=\pi^*\Omega_{g,n}(v_1\otimes...\otimes v_n),\quad 2g-2+n>0
\end{equation}
is a CohFT with {\em flat unit}.  The trivial CohFT $\Omega_{g,n}=1\in H^0(\overline{\modm}_{g,n})$ is a CohFT with flat unit, and if $\Omega$ is a CohFT with flat unit, then $\Omega^{\text{top}}$ is a CohFT with flat unit.  In place of \eqref{flatunit}, a degenerate CohFT may satisfy, for some $\un\in V$,
\begin{equation}  \label{dilunit}  \tag{\theequation${}^\prime$}
\Omega_{g,n+1}(v_1\otimes...\otimes v_n\otimes \un)=\psi_{n+1}\cdot\pi^*\Omega_{g,n}(v_1\otimes...\otimes v_n),\quad 2g-2+n>0.
\end{equation}
The degenerate CohFT $\{\Theta_{g,n}\}$ satisfies \eqref{dilunit}.  

\begin{definition}  \label{cohftheta}
For any CohFT $\Omega$ defined on $(V,\eta)$, define the degenerate CohFT $\Omega^\Theta=\{\Omega^\Theta_{g,n}\}$, also defined on $(V,\eta)$, to be the sequence of $S_n$-equivariant maps $\Omega^\Theta_{g,n}:V^{\otimes n}\to H^*(\overline{\modm}_{g,n})$ given by 
$$\Omega^\Theta_{g,n}(v_1\otimes...\otimes v_n)=\Theta_{g,n}\cdot\Omega_{g,n}(v_1\otimes...\otimes v_n).$$   
\end{definition}
It is immediate that if $\Omega$ is a CohFT satisfying \eqref{flatunit} then $\Omega^\Theta$ is a degenerate CohFT satisfying \eqref{dilunit}.  The operation in Definition~\ref{cohftheta} is an analogue of the tensor product of two CohFTs $\Omega_i$, respectively defined on $(V_i,\eta_i)$ for $i=1,2$, which produces a CohFT defined on $(V_1\otimes V_2,\eta_1\otimes \eta_2)$.  It is given by $\Omega^\Theta=\Omega\otimes\Theta$.  The tensor product generalises the special case of Gromov-Witten invariants of target products and the K\"unneth formula $H^*(X_1\times X_2)\cong H^*(X_1)\otimes H^*(X_2)$.  

A CohFT is {\em semisimple}  if the product defined in \eqref{prod} is semisimple, i.e. if the diagonal decompositition $V\cong\bc\oplus\bc\oplus...\oplus\bc$ respects the product, or equivalently there is a canonical basis $\{ u_1,...,u_D\}\subset V$ such that $u_i\cdot u_j=\delta_{ij}u_i$.  Note that the metric is necessarily diagonal with respect to the canonical basis, $\eta(u_i,u_j)=\delta_{ij}\eta_i$ for some $\eta_i\in\bc \setminus \{0\}$, $i=1,...,D$.    A semisimple topological field theory $\Omega^{\text{top}}$ defined on a vector space $V$ is characterised by an element $\Omega^{\text{top}}_{0,1}\in V^*$ which represents evaluation of the metric $\eta(u_i,u_i)=\eta(u_i\cdot u_i,\un)=\eta(u_i,\un)=\Omega^{\text{top}}_{0,1}(u_i)$.  

\subsection{Gromov-Witten invariants}  \label{GW}
The Gromov-Witten invariants of a target variety $X$ are defined via intersection theory on the moduli space $\overline{\modm}_{g,n}(X,{\bf d})$ of degree ${\bf d}\in H_2(X)$ stable maps $f:(C,x_1,\dots,x_n)\rightarrow X$ of genus $g$ curves with $n$ labeled points into $X$.  Over the moduli space of stable maps are line bundles $L_i\to\overline{\modm}_{g,n}(X,{\bf d})$, for $i=1,...,n$, with fibre the cotangent line $T^*_{x_i}C$ of the $i$th labeled point on the domain curve, and $\psi_i=c_1(L_i)$.  There are natural maps
$$\begin{array}{ccc}\overline{\modm}_{g,n}(X,{\bf d})&\stackrel{ev_i}{\to}&X\\
p\downarrow\quad&&\\ \quad\overline{\modm}_{g,n}\end{array}
$$
given by evaluation $ev_i$ of $f$ at $p_i$, and the forgetful map $p$ which maps a stable map to its domain curve followed by contraction of unstable components.   For $g_1+g_2=g$, ${\bf d}_1+{\bf d}_2={\bf d}$ and $I\sqcup J=\{1,...,n\}$, define $D(g_1,I,{\bf d}_1\mid g_2,J,{\bf d}_2)\subset\overline{\modm}_{g,n}(X,{\bf d})$ to be the image of the map
$$\overline{\modm}_{g_1,|I|+1}(X,{\bf d}_1)\times_X\overline{\modm}_{g_2,|J|+1}(X,{\bf d}_2)\to\overline{\modm}_{g,n}(X,{\bf d})
$$
where the fibre product is over the evaluation map from the final point in each domain, say $x_{|I|+1}$ and $y_{|J|+1}$.  Points of $D(g_1,I,{\bf d}_1\mid g_2,J,{\bf d}_2)$ consist of stable maps with reducible domain.  A special case is when points $p_i$ and $p_{n+1}$ try to come together, denoted by $D_i=D(g,\{1,..,\hat{i},..,n\},{\bf d}\mid 0,\{i,n+1\},0)$.

The Gromov-Witten invariants are defined via intersection with a virtual fundamental class \cite{BFaInt}:
\begin{equation} \label{des}
\left\langle \prod_{i=1}^n\tau_{b_i}(\alpha_i) \right\rangle ^{g}_{{\bf d}}:=\int_{[{\modm}_{g,n}(X,{\bf d})]^{vir}} \prod_{i=1}^n\psi_i^{b_i}ev_i^\ast(\alpha_i)
%\left\langle \Theta\cdot\prod_{i=1}^l\tau_{b_i}(1)  \prod_{i=l+1}^n\tau_{b_i}(\omega)\right\rangle ^{g}_{d}:=\int_{[\overline{\modm}^g_n(\bp^1,d)]^{vir}} \Theta^X_{g,n}\cdot\prod_{i=1}^l\psi_i^{b_i} \prod_{i=l+1}^n\psi_i^{b_i}ev_i^\ast(\omega)
\end{equation}
for $\alpha_i\in H^*(X)$.  If one includes divisor classes into the integral then the Gromov-Witten invariants reduce to Gromov-Witten invariants over lower strata.  See \cite{GatGro} for a proof of the following:
\begin{equation} \label{div}
\int_{[{\modm}_{g,n}(X,{\bf d})]^{vir}} \prod_{i=1}^n\psi_i^{b_i}ev_i^\ast(\alpha_i)\cdot [D(g_1,I,{\bf d}_1\mid g_2,J,{\bf d}_2)]^{vir}\hspace{-1mm}=\sum_{k=1}^N\left\langle\hspace{-.5mm} \tau_0(e_k) \prod_{i\in I}\tau_{b_i}(\alpha_i) \right\rangle ^{g_1}_{{\bf d}_1}\hspace{-2mm}\left\langle\hspace{-.5mm} \tau_0(e^k)\prod_{j\in J}\tau_{b_j}(\alpha_j) \right\rangle ^{g_2}_{{\bf d}_2}
\end{equation}
where $\{e_k\}$ is a basis of $H^{\text{even}}(X;\bc)$ and $\{e^k\}$ is a dual basis using the metric $\eta$.

Define the partition function with respect to a basis $\{e_{\alpha}\}$ of the even part of the cohomology $H^{\text{even}}(X;\bc)$, with $e_0=1\in H^0(X)$, by $Z_X(\h,\{t_k^{\alpha}\})=\exp F_X(\h,\{t_k^{\alpha}\})$ for
$$F_X(\h,\{t_k^{\alpha}\})=
%\sum_{g\geq 0} \h^{2g-2} F_g=  
\sum_{g,{\bf d}}\h^{2g-2}q^{\bf d}\left\langle\exp\left\{\sum_{\alpha,k}\tau_k(e_{\alpha})t_k^{\alpha}\right\}\right\rangle^g_{\bf d}+\langle\tau_1(1)\rangle^1_0\log\h.
$$
and $q^{\bf d}$ is a formal variable living in a Novikov ring which is a power series completion of the semigroup
ring of degrees.  We set $q=1$ when $X=\bp^1$, which we can do more generally in the Fano case.
The presence of the $\log\h$ term is justified by the following {\em dilaton equation}.
$$
\left\langle \tau_1(1)\cdot\prod_{i=1}^n\tau_{b_i}(\alpha_i) \right\rangle ^{g}_{{\bf d}}=(2g-2+n)\left\langle \prod_{i=1}^n\tau_{b_i}(\alpha_i) \right\rangle ^{g}_{{\bf d}}
$$
which can be written as a PDE
$$ \frac{\partial}{\partial t^0_1}Z_X(\h,\{t^{\alpha}_k\})=\left(\h\frac{\partial}{\partial\h}+\sum_{\alpha,k}t^{\alpha}_k\frac{\partial}{\partial t^{\alpha}_k}\right)Z_X(\h,\{t^{\alpha}_k\}).
$$
The $\langle\tau_1(1)\rangle^1_0\log\h$ term is needed so that the PDE above holds in the presence of the $\langle\tau_1(1)\rangle^1_0t^0_1$ term.  Note that if we change coordinates $t^0_1=1+q$, the so-called {\em dilaton shift}, then the dilaton equation is the statement that $Z_X(\h,\{t^{\alpha}_k\})$ is homogeneous of degree zero in $\h$, $q$ and $\{t^{\alpha}_k,\ (\alpha,k)\neq (0,1)\}$.

There is an analogous dilaton equation satisfied by the $\Theta$-Gromov-Witten invariants where $\tau_1(1)$ is replaced by $\tau_0(1)$.
\begin{proposition}   \label{diltheta}
$$
\left\langle \tau_0(1)\cdot\Theta\cdot\prod_{i=1}^n\tau_{b_i}(\alpha_i) \right\rangle ^{g}_{{\bf d}}=(2g-2+n)\left\langle \Theta\cdot\prod_{i=1}^n\tau_{b_i}(\alpha_i) \right\rangle ^{g}_{{\bf d}}
$$
\end{proposition}
\begin{proof}
The proof is analogous to the proof of the dilaton equation---see  for example \cite{GatGro}.  For $2g-2+n>0$, define 
$$\overline{\psi}_i=p^*\psi_i\in H^2(\overline{\modm}_{g,n}(X,{\bf d}))$$
which is used in \eqref{anc} to define ancestor Gromov-Witten invariants.  The following square is commutative,    
$$\begin{array}{ccc}\overline{\modm}_{g,n+1}(X,{\bf d})&\stackrel{\Pi}{\rightarrow}&\overline{\modm}_{g,n}(X,{\bf d})\\
p\downarrow\quad&&p\downarrow\\
\overline{\modm}_{g,n+1}&\stackrel{\pi}{\rightarrow}&\quad\overline{\modm}_{g,n}
\end{array}
$$
where the forgetful map upstairs is denoted by $\Pi$. %, hence $\Pi^*p^*=p^*\pi^*$ in cohomology.
The pull-back property \eqref{forget} satisfied by $\Theta_{g,n}$ downstairs implies the pull-back property upstairs
$$\Theta^X_{g,n+1}=\overline{\psi}_{n+1}\cdot\Pi^*\Theta^X_{g,n}
$$
since
$$\overline{\psi}_{n+1}\cdot\Pi^*\Theta^X_{g,n}=\overline{\psi}_{n+1}\cdot\Pi^*p^*\Theta_{g,n}=\overline{\psi}_{n+1}\cdot p^*\pi^*\Theta_{g,n}=p^*(\psi_{n+1}\cdot \pi^*\Theta_{g,n})=p^*\Theta_{g,n+1}=\Theta^X_{g,n+1}.$$
The dilaton equation is satisfied downstairs, and upstairs by both descendant and ancestor Gromov-Witten invariants.  Downstairs, we have $\pi_*\psi_{n+1}=2g-2+n$ from integration of $\psi_{n+1}$ along a fibre giving the first Chern class of the cotangent bundle.  Upstairs, integration along the fibre gives $\Pi_*\psi_{n+1}=2g-2+n$ and also
$\Pi_*\overline{\psi}_{n+1}=\Pi_*p^*\psi_{n+1}=p^*\pi_*\psi_{n+1}=p^*(2g-2+n)=2g-2+n$.  Hence by the axiom proven in \cite{BehGro} $[\overline{\modm}_{g,n+1}(X,{\bf d})]^{\text{vir}}=\Pi^*[\overline{\modm}_{g,n}(X,{\bf d})]^{\text{vir}}$ we have
\begin{equation}  \label{pushpsi}
\Pi_*(\psi_{n+1}\cdot[\overline{\modm}_{g,n+1}(X,{\bf d})]^{\text{vir}})=(2g-2+n)[\overline{\modm}_{g,n}(X,{\bf d})]^{\text{vir}}=\Pi_*(\overline{\psi}_{n+1}\cdot[\overline{\modm}_{g,n+1}(X,{\bf d})]^{\text{vir}}).
\end{equation}

Downstairs, we have $\psi_{n+1}\cdot D_i=0$ for $D_i$ the $i$th section of the universal curve (given by the ${\bf d}=0$ case of $D_i$ defined above).  Similarly, upstairs
\begin{equation}  \label{psint}
\psi_{n+1}\cdot[D_i]^{\text{vir}}=0=\overline{\psi}_{n+1}\cdot[D_i]^{\text{vir}},\quad \text{for\ } i=1,...,n.
\end{equation}
The first equality in \eqref{psint}, uses the fact that the restriction of $\psi_{n+1}$ to $D_i$ factors via its restriction to $\overline{\modm}_{0,3}$ and hence vanishes.  The second equality in \eqref{psint} uses the first equality in \eqref{psint} together with the relation proven in \cite{GatGro} applied to $i=n+1$
\begin{equation}  \label{psidiv}
\psi_{i}=\overline{\psi}_{i}+\sum_{{\bf d}_1+{\bf d}_2={\bf d}}[D(g,\{1,..,\hat{i},..n\},{\bf d}_1\mid 0,\{i\},{\bf d}_2)]^{\text{vir}}.
\end{equation} 
and the vanishing
$[D(g,\{1,..,n\},{\bf d}_1\mid 0,\{n+1\},{\bf d}_2)]^{\text{vir}}\cdot[D_i]^{\text{vir}}=0$.

Downstairs we have $\psi_i=\pi^*\psi_i+s_i(\overline{\modm}_{g,n})$ and upstairs we have
$\psi_i=\Pi^*\psi_i+s_i(\overline{\modm}_{g,n}(X,{\bf d}))$, where $s_i$ is the $i$th section of the universal curve for $i\in\{1,...,n\}$, or more precisely
\begin{equation}  \label{ancdesclass}
\psi_i\cdot[\overline{\modm}_{g,n+1}(X,{\bf d})]^{\text{vir}}=\Pi^*\psi_i\cdot[\overline{\modm}_{g,n+1}(X,{\bf d})]^{\text{vir}}+[D_i]^{\text{vir}}.
\end{equation}
Since $\psi_i\cdot[D_i]^{\text{vir}}=0$ most terms in powers of this relation vanish to give the following from \cite{GatGro}.  For any $\omega\in H^*(\overline{\modm}_{g,n+1}(X,{\bf d}))$ there are classes $\eta_i\in H^*(\overline{\modm}_{g,n}(X,{\bf d}))$ such that
\begin{align}   \label{push}
\Pi_*\left(\prod_{i=1}^n\psi_i^{b_i}ev_i^\ast(\alpha_i)\cdot\omega\cdot[\overline{\modm}_{g,n+1}(X,{\bf d})]^{\text{vir}}\right)&\\
=\prod_{i=1}^n\psi_i^{b_i}ev_i^\ast(\alpha_i)\cdot&\Pi_*(\omega\cdot[\overline{\modm}_{g,n+1}(X,{\bf d})]^{\text{vir}})+\sum_{i=1}^n \eta_i\cdot\Pi_*(\omega\cdot[D_i]^{\text{vir}}).\nonumber
\end{align}
The proof of the dilaton equation for usual Gromov-Witten invariants applies \eqref{push} to $\omega=\psi_{n+1}$.  Instead we apply \eqref{push} to $\omega=\Theta^X_{g,n+1}$, which acts in some ways like $\psi_{n+1}$.  In particular, in both cases the final term of \eqref{push} vanishes, i.e. by \eqref{psint} 
$$\Theta^X_{g,n+1}\cdot[D_i]^{\text{vir}}=\Pi^*\Theta^X_{g,n}\cdot\overline{\psi}_{n+1}\cdot[D_i]^{\text{vir}}=0.$$ 
 %Perhaps this follows from the same argument as above
This vanishing together with \eqref{pushpsi} gives
\begin{align*}
\Pi_*\left(\prod_{i=1}^n\psi_i^{b_i}ev_i^\ast(\alpha_i)\cdot\Theta^X_{g,n+1}\cdot[\overline{\modm}_{g,n+1}(X,{\bf d})]^{\text{vir}}\right)&=\prod_{i=1}^n\psi_i^{b_i}ev_i^\ast(\alpha_i)\cdot\Pi_*(\Theta^X_{g,n+1}\cdot[\overline{\modm}_{g,n+1}(X,{\bf d})]^{\text{vir}})\\
&=\prod_{i=1}^n\psi_i^{b_i}ev_i^\ast(\alpha_i)\cdot\Theta^X_{g,n}\cdot\Pi_*(\overline{\psi}_{n+1}\cdot[\overline{\modm}_{g,n+1}(X,{\bf d})]^{\text{vir}})\\
&=(2g-2+n)\prod_{i=1}^n\psi_i^{b_i}ev_i^\ast(\alpha_i)\cdot\Theta^X_{g,n}\cdot[\overline{\modm}_{g,n}(X,{\bf d})]^{\text{vir}}.
\end{align*}
Integrate this relation to get the statement of the proposition.
\end{proof}
\begin{remark}
When $X=\{pt\}$ the dilaton equation is
$$
\left\langle \Theta_{g,n+1}\cdot\tau_0\cdot\prod_{i=1}^n\tau_{b_i} \right\rangle ^{g}=(2g-2+n)\left\langle \Theta_{g,n}\cdot\prod_{i=1}^n\tau_{b_i} \right\rangle ^{g}
$$
which is an immediate consequence of $\Theta_{g,n+1}=\psi_{n+1}\cdot\pi^*\Theta_{g,n}$.  We see that $\tau_1$ in the usual dilaton equation is replaced by $\tau_0$ essentially because $\Theta_{g,n+1}$ has a factor of $\psi_{n+1}$, and this leads to $\tau_0(1)$ in place of $\tau_1(1)$ for general $X$.
\end{remark}

The dilaton equation allows us to define the genus zero 1-point invariants by:
\begin{equation}  \label{def01}
\left\langle\Theta\cdot\tau_k(e_{\alpha})\right\rangle^0_{\bf d}:=-\left\langle\Theta\cdot\tau_0(1)\tau_k(e_{\alpha})\right\rangle^0_{\bf d}=-\left\langle\tau_0(1)\tau_k(e_{\alpha})\right\rangle^0_{\bf d}
\end{equation}
which uses the genus zero 2-point usual Gromov invariants.  It would be desirable to have a geometric definition of these genus zero 1-point invariants.  The dimension of the moduli space of stable maps in the usual case is given by $\dim[\overline{\modm}_{0,1}(\bp^1,{\bf d})]^{\text{vir}}=2d-1$, while for the $\Theta$-Gromov invariants we would expect to construct a space of maps of dimension given by $\dim[\modm^{\Theta}_{0,1}(\bp^1,{\bf d})]^{\text{vir}}=2d$.

Define the partition function with respect to a basis $\{e_{\alpha}\}$ of the even part of the cohomology $H^{\text{even}}(X;\bc)$, by $Z^{\Theta}_X(\h,\{t_k^{\alpha}\})=\exp F^{\Theta}_X(\h,\{t_k^{\alpha}\})$ for
\begin{equation}  \label{freenerg}
F^{\Theta}_X(\h,\{t_k^{\alpha}\})=
%\sum_{g\geq 0} \h^{2g-2} F_g=  
\sum_{g,{\bf d}}\h^{2g-2}\left\langle\Theta\cdot\exp\left\{\sum_{\alpha,k}\tau_k(e_{\alpha})t_k^{\alpha}\right\}\right\rangle^g_{\bf d}+\langle\Theta\cdot\tau_0(1)\rangle^1_0\log\h.
\end{equation}
Again $\langle\Theta\cdot\tau_0(1)\rangle^1_0\log\h$ is required in \eqref{freenerg} as a boundary term for the PDE version of the relation in Proposition~\ref{diltheta} given by
\begin{equation}  \label{dilPDE}
 \frac{\partial}{\partial t^0_0}Z^{\Theta}_X(\h,\{t^{\alpha}_k\})=\left(\h\frac{\partial}{\partial\h}+\sum_{\alpha,k}t^{\alpha}_k\frac{\partial}{\partial t^{\alpha}_k}\right)Z^{\Theta}_X(\h,\{t^{\alpha}_k\}).
\end{equation}
If we change coordinates $t^0_0=1+q$ then we see that
$Z^{\Theta}_X(\h,\{t^{\alpha}_k\})$ is homogeneous of degree zero in $\h$, $q$ and $\{t^{\alpha}_k,\ (\alpha,k)\neq (0,0)\}$ which is apparent in the first few terms of \eqref{freenergy}.

The partition functions $Z_X$ and $Z^{\Theta}_X$ can be calculated from the following CohFT and degenerate CohFT constructed out of $X$. 
\begin{definition}  \label{cohftGW}
Associate to a variety $X$ the pair $(V_X,\eta)$ where $V_X=H^{\text{even}}(X;\bc)$, %(and a generalisation of a CohFT on $H^*(X;\bc)$) 
equipped with the metric  
$$\eta(\alpha,\beta)=\int_X\alpha\wedge\beta.$$
Define a CohFT $\Omega_X$ on $(V_X,\eta)$, i.e.
$(\Omega_X)_{g,n}:V_X^{\otimes n}\to H^*(\overline{\modm}_{g,n})$,  by 
$$(\Omega_X)_{g,n}(\gamma_1,...\gamma_n)=\sum_{{\bf d}}q^{\bf d}p_*\left(\prod_{i=1}^nev_i^\ast(\gamma_i)\cap[\overline{\modm}_{g,n}(X,{\bf d})]^{\text{vir}}\right)\in H^*(\overline{\modm}_{g,n})$$
for $\gamma_i\in H^{\text{even}}(X;\bc)$.   In \cite{BehGro,KMaGro} it is proven that $\Omega_X$ is a CohFT with flat unit given by $1\in H^*(X)$.  Note that the dependence of $p=p(g,n,{\bf d})$ on ${\bf d}$ (which is suppressed) results in $(\Omega_X)_{g,n}(\alpha_1,...\alpha_n)$ being a sum of different degree terms. 
\end{definition}
Denote $\hat{Z}_X(\h,\{\bar{t}^{\alpha}_k\})=Z_{\Omega_X}(\h,\{\bar{t}^{\alpha}_k\})$ to be the partition function defined in \eqref{partfun} of the CohFT $\Omega_X$.
$$
\hat{Z}_X(\h,\{\bar{t}^{\alpha}_k\})=\exp\sum_{g,{\bf d}}\h^{2g-2}q^{\bf d}\sum_{n,\vec{\alpha},\vec{k}}\frac{1}{n!}\int_{\overline{\modm}_{g,n}}p_*\left(\prod_{i=1}^nev_i^\ast(e_{\alpha_i})\cap[\overline{\modm}_{g,n}(X,{\bf d})]^{\text{vir}}\right)\cdot\prod_{j=1}^n\psi_j^{k_j}\prod \bar{t}^{\alpha_j}_{k_j}.
$$
The partition functions $Z_X$ and $\hat{Z}_X$ are related by a linear change of variables and multiplication by extra unstable terms---see \eqref{ancdes} below.  We use variables $\bar{t}^{\alpha}_k$ in $\hat{Z}_X$ to facilitate this change of variables.  For $2g-2+n>0$, define {\em ancestor} Gromov-Witten invariants of $X$ by
\begin{equation} \label{anc}
\left\langle \prod_{i=1}^n\overline{\tau}_{b_i}(\alpha_i) \right\rangle ^{g}_{{\bf d}}:=\int_{[{\modm}_{g,n}(X,{\bf d})]^{vir}} \prod_{i=1}^n\overline{\psi}_i^{b_i}ev_i^\ast(\alpha_i)
\end{equation} 
where $\overline{\psi}_i=p^*\psi_i$ replaces $\psi_i$ in \eqref{des}.  The integrand in \eqref{anc} pushes forward to the integrand in the definition of $\hat{Z}_X$.  Define $\hat{Z}_X$ similarly to $Z_X$ by
$$\hat{Z}_X(\h,\{\bar{t}^{\alpha}_k\})=\exp%\sum_{g\geq 0} \h^{2g-2} F_g=  
\sum_{g,{\bf d}}\h^{2g-2}\left\langle\exp\left\{\sum_{\alpha,k}\overline{\tau}_k(e_{\alpha})\bar{t}_k^{\alpha}\right\}\right\rangle^g_{\bf d}
$$
We now refer to the usual Gromov-Witten invariants defined in \eqref{des} as {\em descendant} Gromov-Witten invariants, and $Z_X$, respectively $\hat{Z}_X$, as the descendant, respectively ancestor, partition functions.

The relationship between the descendant and the ancestor invariants uses a sequence of endomorphisms $\{S_k\}$, $k\in \bn$, defined by %endomorphism valued series $S(z)=\sum_{k=0}^{\infty}S_kz^k$ for
\begin{equation}  \label{SKM}
(S_k)^\alpha_\beta=\sum_{\bf d}\langle\tau_0(e^\alpha)\tau_k(e_\beta)\rangle^0_{\bf d}  % can include $q^{\bf d}$ here
\end{equation}
where $e^\alpha=\eta^{\alpha\beta}e_\beta$.  It is proven in \cite{KMaRel} that
\begin{equation}   \label{ancdes}
Z^{st}_X(\h,\{t^{\alpha}_k\})=\left.\hat{Z}_X(\h,\{\bar{t}^{\alpha}_k\})\right|_{ \overline{t}^{\alpha}_k= \sum\limits_{m\geq k}(S_{m-k})^\alpha_\beta t^\beta_m}
\end{equation}
where $Z^{st}_X(\h,\{t^{\alpha}_k\})$ is the stable part defined by
$$Z_X(\h,\{t^{\alpha}_k\})=Z^{st}_X(\h,\{t^{\alpha}_k\})\cdot\exp\h^{-2}\left(\sum_{k,\alpha,{\bf d}}\langle\tau_k(e_\alpha)\rangle^0_{\bf d} t_k^\alpha+\frac12\sum_{\vec{k},\vec{\alpha},{\bf d}}\langle\tau_{k_1}(e_{\alpha_1})\tau_{k_2}(e_{\alpha_2})\rangle^0_{\bf d} t_{k_1}^{\alpha_1}t_{k_2}^{\alpha_2}\right).
$$

Following Definition~\ref{cohftheta}, we define the degenerate CohFT $\Omega_X^{\Theta}$ by
$$(\Omega_X^{\Theta})_{g,n}(\gamma_1,...\gamma_n)=\Theta_{g,n}\cdot\sum_{{\bf d}}p_*\left(\prod_{i=1}^nev_i^\ast(\gamma_i)\right)\in H^*(\overline{\modm}_{g,n}).$$
and its partition function
$$\hat{Z}^{\Theta}_{X}(\h,\{\bar{t}^{\alpha}_k\})=\exp\sum_{\Small{\begin{array}{c}g,n,\vec{k}\\ \vec{\alpha},{\bf d}\end{array}}}\frac{\h^{2g-2}}{n!}\int_{\overline{\modm}_{g,n}}\Theta_{g,n}\cdot p_*\left(\prod_{i=1}^nev_i^\ast(e_{\alpha_i})\right)\cdot\prod_{j=1}^n\psi_j^{k_j}\prod \bar{t}^{\alpha_j}_{k_j}.
$$
The next proposition shows that the relationship \eqref{ancdes} also holds between the descendant and the ancestor $\Theta$-Gromov-Witten invariants, where again the stable part is defined by removing the genus zero 1-point and 2-point functions.
\begin{proposition}  \label{th:ancdest}
The ancestor and descendant $\Theta$-Gromov-Witten invariants are related by:
\begin{equation}   \label{ancdest}
(Z^{\Theta}_X)^{st}(\h,\{t^{\alpha}_k\})=\left.\hat{Z}^{\Theta}_X(\h,\{\bar{t}^{\alpha}_k\})\right|_{ \overline{t}^{\alpha}_k= \sum\limits_{m\geq k}(S_{m-k})^\alpha_\beta t^\beta_m}
\end{equation} 
using the same change of variable as for usual Gromov-Witten invariants.
\end{proposition}
\begin{proof}
The proof of \eqref{ancdes} uses an inductive argument based on the relationship between $\psi_i$ and $\overline{\psi}_i$ in \eqref{psidiv}.  It immediately allows cup product with $\Theta_{g,n}$ which we describe here.

In order to compare insertions of $\psi_i$ and $\overline{\psi}_i$ we need the following generalised $\Theta$-Gromov-Witten invariants involving insertions of both:
\begin{equation} \label{gendes}
\left\langle \Theta\cdot\prod_{i=1}^n\tau_{a_i,b_i}(\alpha_i) \right\rangle ^{g}_{{\bf d}}:=\int_{[{\modm}_{g,n}(X,{\bf d})]^{vir}} \Theta_{g,n}^X\cdot\prod_{i=1}^n\psi_i^{a_i}\overline{\psi}_i^{b_i}ev_i^\ast(\alpha_i).
\end{equation}
This was introduced by Kontsevich and Manin \cite{KMaRel} for the usual Gromov-Witten invariants.  By taking the cup product of $\Theta_{g,n}^X\cdot\prod_{i=1}^n\psi_i^{a_i}\overline{\psi}_i^{b_i}ev_i^\ast(\alpha_i)$ with the relation \eqref{psidiv} and integrating, we have
\begin{align*} 
\left\langle \Theta\cdot\tau_{a_1+1,b_1}(\alpha_1)\prod_{i=2}^n\tau_{a_i,b_i}(\alpha_i) \right\rangle ^{g}_{{\bf d}}=&\left\langle \Theta\cdot\tau_{a_1,b_1+1}(\alpha_1)\prod_{i=2}^n\tau_{a_i,b_i}(\alpha_i) \right\rangle ^{g}_{{\bf d}}\\
&\qquad+
\sum_{k,{\bf d}_2}\left\langle \Theta\cdot\tau_{0,b_1}(e_k)\prod_{i=2}^n\tau_{a_i,b_i}(\alpha_i) \right\rangle ^{g}_{{\bf d}_1}\left\langle \tau_0(e^k)\tau_{a_1,0}(\alpha_1) \right\rangle ^{0}_{{\bf d}_2}
\end{align*}
where ${\bf d}_1+{\bf d}_2={\bf d}$.
The sum over the basis $\{e_k\}$ of $H^{\text{even}}(X;\bc)$ and its dual basis $\{e^k\}$ comes from restriction to the divisor given in \eqref{div}.  Apply this relation to the sequence of insertions $\tau_{m_1}=\tau_{m_1,0}$, $\tau_{m_1-1,1}$,...,$\tau_{0,m_1}=\overline{\tau}_{m_1}$ (and set $b_i=0$ for $i>0$) to get 
$$
\left\langle \Theta\cdot\prod_{i=1}^n\tau_{m_i}(\alpha_i) \right\rangle ^{g}_{{\bf d}}=\left\langle \Theta\cdot\overline{\tau}_{m_1}(\alpha_1)\prod_{i=2}^n\tau_{m_i}(\alpha_i) \right\rangle ^{g}_{{\bf d}}
+\mathop{\sum_{a+b=m_1}}_{k,{\bf d}_2}\left\langle \Theta\cdot\overline{\tau}_a(e_k)\prod_{i=2}^n\tau_{m_i}(\alpha_i) \right\rangle ^{g}_{{\bf d}_1}\left\langle \tau_0(e^k)\tau_b(\alpha_1) \right\rangle ^{0}_{{\bf d}_2}.
$$
This gives exactly the coordinate change in \eqref{ancdest}.
\end{proof}

\subsection{Givental action}   \label{Giv}
In this section we recall the group action on cohomological field theories due to Givental.  The action on a CohFT $\Omega$ is presented as a weighted sum over stable graphs with vertices weighted by $\Omega_{g,n}$, and edges weighted by interesting combinations of $\psi$ classes.  Givental's original action was defined on partition functions of CohFTs.  The action on CohFTs , i.e. the actual cohomology classes in $H^*(\overline{\modm}_{g,n})$, described here was discovered independently, by Katzarkov-Kontsevich-Pantev, Kazarian and Teleman---see \cite{PPZRel,ShaBCOV}.

Consider an element of the loop group $LGL(V)$ given by a formal series
$$R(z)  =  \sum_{k=0}^\infty R_k z^k$$
where $R_k$ are endomorphisms of $V$ and $R_0=Id$.  The twisted loop group $L^{(2)}GL(V)\subset LGL(V)$ is the subgroup defined to consist of elements satisfying
$$ R(z)R(-z)^T=Id.
$$
Choose a basis of $V$, known as the normalised canonical basis, with respect to which $\eta=I$ and the endomorphisms $R(z)$ become matrices.   For $R\in L^{(2)}GL(V)$, define 
\begin{equation}   \label{edgewt}
\ce(w,z)=\frac{I-R^{-1}(z)R^{-1}(w)^T}{w+z}=\sum_{i,j\geq 0}\ce_{ij}w^iz^j
\end{equation}
which has the power series expansion on the right since the numerator $I-R^{-1}(z)R^{-1}(w)^T$ vanishes at $w=-z$ since $R^{-1}(z)$ is also an element of the twisted loop group.

Givental's action is defined via weighted sums over stable graphs.  Dual to any point $(C,p_1,...,p_n)\in\overline{\modm}_{g,n}$ is its stable graph $\Gamma$ with vertices $V(\Gamma)$ representing irreducible components of $C$, internal edges representing nodal singularities and a (labeled) external edge for each $p_i$.  %Each vertex is labeled by a genus $g(v)$ and has valency $n(v)$.  The genus of a stable graph is $g(\Gamma)=\displaystyle h_1(\Gamma)+\hspace{-3mm}\sum_{v\in V(\Gamma)}g(v)$.   
\begin{definition}
Define $G_{g,n}$ to be the set of all stable, connected, genus $g$ graphs with $n$ labeled external leaves, dual to a stable curve $(C,p_1,...,p_n)$.  For any graph $\gamma\in G_{g,n}$ denote by
$$V(\gamma),\quad E(\gamma),\quad H(\gamma),\quad L(\gamma)=L^*(\gamma)\sqcup L^\bullet(\gamma)$$
its set of vertices, edges, half-edges and leaves, or external edges.  Associated to any vertex $v\in V(\gamma)$ is its valence $n_v$ and genus $g_v$.  The leaves $L(\gamma)$ consist of ordinary leaves $L^*$ and dilaton leaves $L^\bullet$.  Associated to any ordinary leaf $\ell\in L^*$ is its label $p(\ell)\in\{1,2,...,n\}$.  The set of half-edges consists of leaves and oriented edges so there is an injective map $L(\gamma)\to H(\gamma)$ and a multiply-defined map $E(\gamma)\to H(\gamma)$ denoted by $E(\gamma)\ni e\mapsto \{e^+,e^-\}\subset H(\gamma)$.
The map sending a half-edge to its vertex is given by $v:H(\gamma)\to V(\gamma)$.  The {\em genus} of $\gamma$ is $g(\gamma)=\displaystyle b_1(\gamma)+\hspace{-2mm}\sum_{v\in V(\gamma)}\hspace{-2mm}g(v)$.  We say $\gamma$ is {\em stable} if any vertex labeled by $g=0$ is of valency $\geq 3$ and there are no isolated vertices labeled by $g=1$.   
\end{definition}
For a given stable graph $\Gamma$ of genus $g$ and with $n$ external edges we have
$$\phi_{\Gamma}:\overline{\modm}_{\Gamma}=\prod_{v\in V(\Gamma)}\overline{\modm}_{g_v,n_v}\to\overline{\modm}_{g,n}.$$

Given a CohFT $\Omega=\{\Omega_{g,n}\in H^*(\overline{\modm}_{g,n})\otimes (V^*)^{\otimes n}\mid g,n\in\bn,2g-2+n>0\}$, following \cite{PPZRel,ShaBCOV} define a new CohFT $R\Omega=\{(R\Omega)_{g,n}\}$ by a weighted sum over stable graphs, with weights defined as follows.
\begin{enumerate}[(i)]
\item {\em Vertex weight:} $w(v)=\Omega_{g_v,n_v}\in (V^*)^{\otimes n_v}\otimes H^*(\overline{\modm}_{g_v,n_v})$ at each vertex $v$ 
\item {\em Leaf weight:} $w(\ell)=R^{-1}(\psi_{p(\ell)})\in End(V)\otimes H^*(\overline{\modm}_{g_{v(\ell)},n_{v(\ell)}})$ at each leaf $\ell$
\item {\em Edge weight:} $w(e)=\ce(\psi_e',\psi_e'')\in V^{\otimes 2}\otimes H^*(\overline{\modm}_{g_{v(e')},n_{v(e')}})\otimes H^*(\overline{\modm}_{g_{v(e'')},n_{v(e'')}})$ at each edge $e$
\end{enumerate}
Then
$$(R\Omega)_{g,n}=\sum_{\Gamma\in G_{g,n}}\frac{1}{|{\rm Aut}(\Gamma)|}(p_{\Gamma})_*\hspace{-6mm}\prod_{\begin{array}{c}v\in V(\Gamma)\\\ell\in L^*(\Gamma)\\ e\in E(\Gamma)\end{array}} \hspace{-4mm}w(v) w(\ell)w(e)
$$
where we contract in the $V$ factor to get $w(v)w(\ell)w(e)\in (V^*)^n\otimes H^*(\overline{\modm}_\Gamma)$. 
This defines an action of the twisted loop group on CohFTs.  
 
Define a translation action of $T(z)\in z^2V[[z]]$ on the sequence $\Omega_{g,n}$ aby
\begin{equation}  \label{transl}
(T\Omega)_{g,n}(v_1\otimes...\otimes v_n)=\sum_{m\geq 0}\frac{1}{m!}p_*\Omega_{g,n+m}(v_1\otimes...\otimes v_n\otimes T(\psi_{n+1})\otimes...\otimes  T(\psi_{n+m}))
\end{equation}
where $p:\overline{\modm}_{g,n+m}\to\overline{\modm}_{g,n}$ is the forgetful map.
Since $T(z)\in z^2V[[z]]$, $\dim\overline{\modm}_{g,n+m}=3g-3+n+m$ grows more slowly in $m$ than the degree $2m$ coming from $T$ which ensures the sum \eqref{transl} is finite.  
\begin{remark}
We can relax the requirement $T(z)\in z^2V[[z]]$ and instead allow $T(z)\in zV[[z]]$ to act via \eqref{transl} but now one must control $\Omega_{g,n}$ as $n\to\infty$ to ensure \eqref{transl} is a finite sum.  Later we will see examples where $\Omega_{g,n}(v_1\otimes...\otimes v_n)\cdot\alpha=0$ for $\deg\alpha\geq g-1$ so in particular the degree $m$ coming from $T(z)\in zV[[z]]$ will eventually annihilate $\Omega_{g,n+m}$ ensuring the sum \eqref{transl} is finite.
\end{remark}
The translation action can be realised graphically via:
\begin{enumerate}[(iv)]
\item {\em Dilaton leaf weight:}  $w(\ell)=T(\psi_{p(\ell)})$ at each dilaton leaf $\ell\in L^\bullet$.
\end{enumerate}

The action of $R(z)$ and $T(z)$ on CohFTs does not require the condition \eqref{unmet} and hence generalises to an action on degenerate CohFTs.  The action on (degenerate) CohFTs immediately gives rise to an action on partition functions, which store correlators of the sequence of cohomology classes.  It gives a graphical construction of the partition functions $Z_{R\Omega}$ and $Z_{T\Omega}$ out of the partition function $Z_{\Omega}$.  

Givental and Teleman \cite{GivGro,TelStr} proved that the twisted loop group acts transitively on semisimple CohFTs.  This gives a way to construct semisimple CohFTs on a vector space of dimension $D$---simply act on $D$ copies of the trivial CohFT.   The partition function of a semisimple CohFT is then constructed as a graphical expansion with vertex contributions given by the basic building block $Z^{\text{KW}}(\h,t_0,t_1,...)$, defined in \eqref{KW}, which represents the trivial CoHFT.      %In the semi-simple case, we usually write $R_k$ as $D\times D$ matrices with respect to the canonical basis. and in particular is determined by the topological, or degree zero, part of the CohFT
The following theorem produces a CohFT $\Omega$ out of a TFT, i.e. a degree zero CohFT, which necessarily coincides with  $\Omega^{\text{top}}$ defined in \eqref{cohfttop}.
\begin{theorem}[\cite{GivGro,TelStr}]  \label{GivTelThm}
Given a semisimple CohFT $\Omega$, there exists $R(z)\in L^{(2)}GL(V)$ and $T(z)\in z^2V[[z]]$ such that
$$\Omega=R\cdot T\cdot\Omega^{\text{top}}.
$$
\end{theorem}
The proofs of Theorem~\ref{GivTelThm} in \cite{GivGro,TelStr} show how to construct the elements $R(z)$ and $T(z)$ out of the CohFT.  We have omitted details here because we will need only some properties of $R(z)$.  In particular the following proposition, which shows how the classes $\Theta_{g,n}$ interact with the group action in Theorem~\ref{GivTelThm}, only needs the existence of $R(z)$.
\begin{proposition}  \label{thetadecomp}
If $\Omega=R\cdot T\cdot\Omega^{\text{top}}$ for $R(z)\in L^{(2)}GL(V)$ and $T(z)\in z^2V[[z]]$ then 
$$\Omega^{\Theta}=R\cdot T_0\cdot\left(\Omega^{\text{top}}\right)^{\Theta}
$$
where $T_0(z)=T(z)/z\in zV[[z]]$.
\end{proposition}
\begin{proof}
Using $\Omega=R\cdot T\cdot\Omega^{\text{top}}$, express $\Omega_{g,n}$ as a weighted sum over genus $g$ stable graphs with $n$ leaves.
$$\Omega_{g,n}=\sum_{\Gamma\in G_{g,n}}\frac{1}{|{\rm Aut}(\Gamma)|}(p_{\Gamma})_*\hspace{-6mm}\prod_{\begin{array}{c}v\in V(\Gamma)\\\ell\in L^*(\Gamma)\\ e\in E(\Gamma)\end{array}} \hspace{-4mm}w(v) w(\ell)w(e)
$$
where the weights are given by $w(v)=T\cdot \Omega_{g_v,n_v}^{\text{top}}$ at each vertex $v$, $w(\ell)=R^{-1}(\psi_{p(\ell)})$ at each leaf $\ell$ and $w(e)=\ce(\psi_e',\psi_e'')$ at each edge $e$.  The polynomial in two variables $\ce$ is defined from $R$ in \eqref{edgewt}.

Hence $\Omega^{\Theta}_{g,n}=\Theta_{g,n}\cdot\Omega_{g,n}$ is a weighted sum over stable graphs.  Since $\Theta_{g,n}$ is a CohFT we have $p_{\Gamma}^*\Theta_{g,n}=\Theta_{\Gamma}$ hence
$$\Theta_{g,n}\cdot(p_{\Gamma})_*\omega=(p_{\Gamma})_*\Theta_{\Gamma}\cdot\omega.
$$
Consider a vertex $v$ of the stable graph $\Gamma$ of genus $g_v$ and valence $n_v$.  Inside the factor $\overline{\modm}_{g_v,n_v}$, the contribution to the sum giving $\Omega^{\Theta}_{g,n}$ is given by a product of $\psi$ classes, determined by edge and leaf weights, times $\Theta_{g_v,n_v}$ (which is part of $\Theta_{\Gamma}$) times $T\cdot\Omega_{g_v,n_v}^{\text{top}}$.    We claim that
$$ \Theta_{g_v,n_v}\cdot T\cdot\Omega_{g_v,n_v}^{\text{top}}=T_0\cdot \Theta_{g_v,n_v}\cdot\Omega_{g_v,n_v}^{\text{top}}.
$$
The translation $T\cdot\Omega_{g_v,n_v}^{\text{top}}$ is realised by inserting the vector coefficients of $T$ into dilaton leaves labeled $n_v+1,...,n_v+m$ and pushing forward powers $\psi_{n_v+i}^{m}$ to $\overline{\modm}_{g_v,n_v}$ via the forgetful map.  The translation $T_0\cdot \Theta_{g_v,n_v}\cdot\Omega_{g_v,n_v}^{\text{top}}$ pushes forward powers $\psi_{n_v+i}^m$ next to $\Theta_{g_v,n_v+m}=\psi_{n_v+1}...\psi_{n_v+m}\pi^*\Theta_{g_v,n_v}$.  The extra factor of $\psi_{n_v+i}$ means the power $\psi_{n_v+i}^m$ becomes $\psi_{n_v+i}^{m+1}$.  This results in a shift by one in the translation producing $T_0(z)=T(z)/z\in zV[[z]]$.  Since $\Theta_{g_v,n_v}\cdot\Omega_{g_v,n_v}^{\text{top}}=(\Omega^{\text{top}})_{g_v,n_v}^{\Theta}$, at each vertex we place $\Theta_{g_v,n_v}$ rescaled by a topological field theory, and the result follows.
\end{proof}
An efficient way to store the decomposition of a semisimple CohFT $\Omega$ given in Theorem~\ref{GivTelThm}, together with the data of $R$ and $T$, is via topological recursion which we describe next.  We will see that Proposition~\ref{thetadecomp} allows us to produce associated topological recursion data that stores the decomposition of $\Omega^{\Theta}$.

\subsection{Topological recursion}   \label{sec:TR}

Topological recursion, as developed by Chekhov, Eynard, Orantin \cite{CEyHer,EOrInv}, is a procedure which takes as input a spectral curve, defined below, and produces a collection of symmetric tensor products of meromorphic 1-forms $\omega_{g,n}$ on $C^n$ which we refer to as correlators.   The correlators store enumerative information in different ways.  Periods of the correlators store top intersection numbers of tautological classes in the moduli space of stable curves $\overline{\modm}_{g,n}$ and local expansions of the correlators can serve as generating functions for enumerative problems.  

A fundamental ingredient in topological recursion is a bidifferential $B(p,p')$ defined on $C\times C$ as follows.
\begin{definition} \label{bidiff} On any compact Riemann surface $C$ with a choice of $\cal A$-cycles $(\Sigma,\{ {\cal A}_i\}_{i=1,...,g})$, define a {\em fundamental normalised bidifferential of the second kind}  $B(p,p')$ to be a symmetric tensor product of differentials on $C\times C$, uniquely defined by the properties that it has a double pole on the diagonal of zero residue, double residue equal to $1$, no further singularities and normalised by $\int_{p\in{\cal A}_i}B(p,p')=0$, $i=1,...,g$, \cite{FayThe}.  On a rational curve, which is sufficient for this paper, $B$ is the Cauchy kernel
$$
B(z_1,z_2)=\frac{dz_1dz_2}{(z_1-z_2)^2}.$$  
\end{definition}
The bidifferential $B(p,p')$ acts as a kernel for producing meromorphic differentials on the Riemann surface $C$ via $\omega(p)=\int_{\Lambda}\lambda(p')B(p,p')$ where $\lambda$ is a function defined along the contour $\Lambda\subset C$.  
Topological recursion produces tensor products of differentials built from $B(p,p')$ such as the normalised (trivial $\cal A$-periods) differentials of the second kind, holomorphic outside the zeros of $dx$ and with a double pole at a simple zero of $dx$ used in Definition~\ref{evaluationform}.

A spectral curve $S=(C,x,y,B)$ is a Riemann surface $C$ equipped with two meromorphic functions $x, y: C\to \mathbb{C}$ and a bidifferential $B(p_1,p_2)$ defined in Definition~\ref{bidiff}, which is the Cauchy kernel in this paper.   Topological recursion is a procedure that produces from a spectral curve $S=(C,x,y,B)$ a symmetric tensor product of meromorphic 1-forms $\omega_{g,n}$ on $C^n$ for integers $g\geq 0$ and $n\geq 1$, which we refer to as {\em correlation differentials} or {\em correlators}.  The correlation differentials $\omega_{g,n}$ are defined by
\[
\omega_{0,1}(p_1) = -y(p_1) \, d x(p_1) \qquad \text{and} \qquad \omega_{0,2}(p_1, p_2) = B(p_1,p_2) %\frac{d p_1 \otimes d p_2}{(p_1-p_2)^2}.
\]
and for $2g-2+n>0$ they are defined recursively via the following equation.
\[
\omega_{g,n}(p_1, p_L) = \sum_{d x(\alpha) = 0} \mathop{\text{Res}}_{p=\alpha} K(p_1, p) \Bigg[ \omega_{g-1,n+1}(p, \hat{p}, p_L) + \mathop{\sum_{g_1+g_2=g}}_{I \sqcup J = L}^\circ \omega_{g_1,|I|+1}(p, p_I) \, \omega_{g_2,|J|+1}(\hat{p}, p_J) \Bigg]
\]
Here, we use the notation $L = \{2, 3, \ldots, n\}$ and $p_I = \{p_{i_1}, p_{i_2}, \ldots, p_{i_k}\}$ for $I = \{i_1, i_2, \ldots, i_k\}$. The outer summation is over the zeroes $\alpha$ of $dx$ and $p \mapsto \hat{p}$ is the involution defined locally near $\alpha$ satisfying $x(\hat{p}) = x(p)$ and $\hat{p} \neq p$. The symbol $\circ$ over the inner summation means that we exclude any term that involves $\omega_{0,1}$. Finally, the recursion kernel is given by
\[
K(p_1,p) = \frac{1}{2}\frac{\int_{\hat{p}}^p \omega_{0,2}(p_1, \,\cdot\,)}{[y(p)-y(\hat{p})] \, d x(p)}.
\]
which is well-defined in the vicinity of each zero of $dx$.   It acts on differentials in $p$ and produces differentials in $p_1$ since the quotient of a differential in $p$ by the differential $dx(p)$ is a meromorphic function.  For $2g-2+n>0$, each $\omega_{g,n}$ is a symmetric tensor product of meromorphic 1-forms on $C^n$ with residueless poles at the zeros of $dx$ and holomorphic elsewhere.  A zero $\alpha$ of $dx$ is {\em regular}, respectively irregular, if $y$ is regular, respectively has a simple pole, at $\alpha$.  The order of the pole in each variable of $\omega_{g,n}$ at a regular, respectively irregular, zero of $dx$ is $6g-4+2n$, respectively $2g$, \cite{DNoTop}.  Define $\Phi(p)$ up to an additive constant by $d\Phi(p)=y(p)dx(p)$.  For $2g-2+n>0$, the invariants satisfy the dilaton equation~\cite{EOrInv}
\[
\sum_{\alpha}\Res_{p=\alpha}\Phi(p)\, \omega_{g,n+1}(p,p_1, \ldots ,p_n)=(2g-2+n) \,\omega_{g,n}(p_1, \ldots, p_n),
\] 
where the sum is over the zeros $\alpha$ of $dx$.  This enables the definition of the so-called {\em symplectic invariants}
\[ F_g=\sum_{\alpha}\Res_{p=\alpha}\Phi(p)\omega_{g,1}(p).\]

\begin{remark}
A generalisation of topological recursion known as {\em local topological recursion} allows the Riemann surface $C$ to be non-compact and disconnected.  In this case the bidifferential $B(p_1,p_2)$ is simply a prescribed symmetric tensor product of differentials on $C\times C$, uniquely defined by the properties that it has a double pole on the diagonal of zero residue, double residue equal to $1$ and no further singularities.  This generalisation arises out of the fact that topological recursion on a spectral curve $(C,x,y,B)$ where $C$ is a compact Riemann surface can be formulated by restriction of $(x,y,B)$ to the non-compact Riemann surface $\tilde{C}$ given by the disjoint union of small open neighbourhoods of the zeros of $dx$.
\end{remark}

\begin{definition}\label{evaluationform}
For a Riemann surface equipped with a meromorphic function $(\Sigma,x)$ we define evaluation of any meromorphic differential $\omega$ at a simple zero $\cp$ of $dx$ by
$$
\omega(\cp)^2:=\Res_{p=\cp}\frac{\omega(p)\omega(p)}{dx}.%\Res_{p=\cp}\frac{\omega(p)}{\sqrt{2(x(p)-x(\cp))}}%=\left.\frac{\omega(p)}{ds}\right|_{s=0}.
$$
This defines $\omega(\cp)\in\bc$ up to a $\pm1$ ambiguity which can be removed by making a choice of square root.% where we choose a branch of $\sqrt{x(p)-x(\cp)}$ once and for all at $\cp$ to remove .
\end{definition}
The correlators $\omega_{g,n}$ are normalised differentials of the second kind in each variable---they have zero $\ca$-periods, poles of zero residue at the zeros $\alpha$ of $dx$, and holomorphic elsewhere.  Their principal parts are skew-invariant under the local involution $p\mapsto\hat{p}$.  A basis of such normalised differentials of the second kind is constructed from $x$ and $B$ in the following definition. 
\begin{definition}\label{auxdif}
For a Riemann surface $C$ equipped with a meromorphic function $x:C\to\bc$ and bidifferential $B(p_1,p_2)$ define the auxiliary differentials on $C$ as follows.  For each zero $\alpha_i$ of $dx$,  define
\begin{equation}  \label{Vdiff}
U^i_0(p)=B(\alpha_i,p),\quad U^i_{k+1}(p)=d\left(\frac{U^i_k(p)}{dx(p)}\right),\ i=1,...,D,\quad k=0,1,2,...
\end{equation}
where evaluation $B(\alpha_i,p)$ at  $\alpha_i$ is given in Definition~\ref{evaluationform}.
\end{definition}
The correlators $\omega_{g,n}$ are polynomials in the auxiliary differentials $U^i_k(p)$.  To any spectral curve $S$, one can define a partition function $Z^S$ by assembling the polynomials built out of the correlators $\omega_{g,n}$ \cite{DOSSIde,EynInv}.
\begin{definition}  \label{TRpart}
$$Z^S(\h,\{u^{\alpha}_k\}):=\left.\exp\sum_{g,n}\frac{\h^{2g-2}}{n!}\omega^S_{g,n}\right|_{U^{\alpha}_k(p_i)=u^{\alpha}_k}.
$$
\end{definition}
%As usual define $F_g$ to be the contribution from $\omega_{g,n}$: $$\log Z^S(\h,\{u^{\alpha}_k\})=\sum_{g\geq 0}\h^{g-1}F_g^S(\{u^{\alpha}_k\}).$$

\subsubsection{Topological recursion and the Givental action.}  
It was proven in \cite{DOSSIde} that for any semisimple CohFT with flat unit $\Omega$, there exists a (local) spectral curve $S$ which produces the partition function $Z_\Omega$ via $Z^S(\hbar,\{u^\alpha_k\})=Z_\Omega(\hbar,\{u^\alpha_k\})$.  As we saw in Theorem~\ref{GivTelThm}, a semisimple CohFT with flat unit $\Omega$ is equivalent to three pieces of information: $R(z)\in L^{(2)}GL(V,\bc)$, $T(z)\in z^2V[[z]]$ and $\eta_1,...,\eta_{\dim V}\in\bc$ which encodes a topological field theory via $\eta(u_i,u_j)=\delta_{ij}\eta_i$ for $\{u_i\}$ a canonical basis of $V$.  We will describe below how $R(z)$ is obtained from $x(p)$ and the bidifferential $B(p_1,p_2)$, and $T(z)$ and the topological field theory are obtained from $x(p)$ and $y(p)$.  The role of the function $x(p)$ is to produce local coordinates with respect to which local expansions of $B(p_1,p_2)$ and $y(p)$ give $R(z)$ and $T(z)$.

An element of the twisted loop group $R(z)\in L^{(2)}GL(D,\bc)$ can be naturally defined from a Riemann surface $\Sigma$ equipped with a bidifferential $B(p_1,p_2)$ on $\Sigma\times\Sigma$ and a meromorphic function $x:\Sigma\to\bc$, where $D$ is the number of zeros of $dx$.  Define
\begin{equation}   \label{RmatB}
\left[R_B^{-1}(z)\right]^i_j = -\frac{\sqrt{z}}{\sqrt{2\pi}}\int_{\Gamma_j} B(\alpha_i,p)\cdot e^{\frac{(x(\alpha_j)-x(p))}{ z}}
\end{equation}
where $\Gamma_j$ is a path of steepest descent of $-x(p)/z$ containing $\alpha_j$.  More precisely, $[R^{-1}(z)]^i_j$ is an asymptotic expansion of the integral as $z\to 0$, which depends only on a neighbourhood of $p=\alpha_j$.  The proof that $R(z)R^T(-z)=I$ can  be found in \cite{DNOPSPri,EynInv}.   A basic example is the function $x=z^2$ on $\Sigma=\bc$ which gives rise to the constant element $R(z)=1\in GL(1,\bc)$.  More generally, since any function $x$ looks like this example locally $R(z)=I+R_1z+...\in L^{(2)}GL(D,\bc)$ is a deformation of $I\in GL(D,\bc)$.
%Given a spectral curve $S=(\Sigma,B,x,y)$ define $R(z)=\sum R_kz^k\in\text{End}(V)[[z]]$ for a dimension $D(=$ number of zeros of $dx$) vector space $V$ by

The local expansion of a regular function $y(p)$ at $p=\alpha_i$ gives the translation $T^y$ and semisimple TFT $\Omega^{\text{top}}_y$ defined by
\begin{equation}   \label{transy}
T^y(z)=-\sum_{k=2}^{\infty} (2k-1)!!\left\{\frac{y^{\alpha}_{2k-1}}{y^{\alpha}_{1}}\right\}z^{k}\in z^2V[[z]]\qquad \Omega^{\text{top},y}_{0,1}=\{y_{1,\alpha}\}\in V^*
\end{equation}
where $y_{2k-1,\alpha}$ are the odd coefficients of the local expansion $y=\sum_{j\geq 0} y_{j,\alpha}s^j$ with respect to the local coordinate $s$ defined by $x=\frac12 s^2+x(\alpha_i)$.  Here we have encoded a semisimple TFT by $\Omega^{\text{top}}_{0,1}\in V^*$ as described in Section~\ref{sec:cohft}.

\begin{theorem}[\cite{DOSSIde}]  \label{DOSS}
For any semisimple CohFT with flat unit $\Omega=R\cdot T\cdot\Omega^{\text{top}}$, there exists a local spectral curve $S=(\Sigma,B,x,y)$ with $R(z)=R_B(z)$ defined in \eqref{RmatB}, $T(z)=T^y(z)$ and $\Omega^{\text{top}}=\Omega^{\text{top},y}$ defined in \eqref{transy}, such that
$$Z^S(\hbar,\{u^\alpha_k\})=Z_\Omega(\hbar,\{u^\alpha_k\}).$$
\end{theorem}
\begin{remark}
The flat unit condition is rather strong and in particular it means that the function $y(p)$, hence $T^y(z)$, is uniquely determined from the $D$-tuple $\Omega^{\text{top}}_y=\{y_{1,\alpha}\}$ and $R(z)$.   Dubrovin associated to any semisimple CohFT with flat unit a family of spectral curves known as a superpotential \cite{DubPai}.   In \cite{DNOPSDub} it is proven that the superpotential can be used to produce a compact spectral curve out of Theorem~\ref{DOSS}.
\end{remark}
Theorem~\ref{DOSS} can be generalised to allow for CohFTs without flat unit and degenerate CohFTs which correspond to irregular spectral curves.  The translation associated to an irregular point is 
\begin{equation}   \label{transyi}
T_0^y(z)=-\sum_{k=1}^{\infty} (2k-1)!!\left\{\frac{y^{\alpha}_{2k-1}}{y^{\alpha}_{-1}}\right\}z^{k}\in zV[[z]]\qquad \Omega^{\text{top},y}_{0,1}=\{y_{-1,\alpha}\}\in V^*
\end{equation}
where $y_{2k-1,\alpha}$ are the odd coefficients of the local expansion $\displaystyle y=\sum_{j\geq -1} y_{j,\alpha}s^j$ with respect to the local coordinate \vspace{-4mm}\\ 
$s$ defined by $x=\frac12 s^2+x(\alpha_i)$.

\begin{theorem}[\cite{CNoTop}]  \label{CNo}
For any degenerate CohFT $\Omega=R\cdot T_0\cdot(\Omega^{\text{top}})^{\Theta}$, there exists a local spectral curve $S=(\Sigma,B,x,y)$ with $R(z)=R_B(z)$ defined in \eqref{RmatB}, $T_0(z)=T^y(z)$ and $\Omega^{\text{top}}=\Omega^{\text{top},y}$ defined in \eqref{transyi}, such that
$$Z^S(\hbar,\{u^\alpha_k\})=Z_\Omega(\hbar,\{u^\alpha_k\}).$$
\end{theorem}
\begin{remark}
The expression $R\cdot T_0\cdot(\Omega^{\text{top}})^{\Theta}$ gives a weighted sum over graphs, as described in Section~\ref{Giv}, where a vertex of type $(g',n')$ is weighted by the cohomology class $\Theta_{g',n'}$ and the TFT, whereas the expression $R\cdot T\cdot\Omega^{\text{top}}$ weights each vertex by the trivial cohomology class $1$ and the TFT.  One can of course mix the two---assign $\Theta_{g',n'}$ to some vertices and $1$ to others.  Indeed, a more general version of Theorem~\ref{CNo} is proven in \cite{CNoTop} which allows mixed vertex contributions.
\end{remark}
 
%Virasoro relations involving only $s_k$ variables found in \cite{AMoHer}.  Perhaps equivalent to TR.

%One can use in place of $Z^{\text{KW}}(\h,t_0,t_1,...)$ the building block $Z^{\text{BGW}}(\h,t_0,t_1,...)$.
 
\section{Gromov-Witten invariants of $\bp^1$}  \label{proofs}
From now on we consider Gromov-Witten invariants of $\bp^1$.  The main outcome of this section is the calculation of the $\Theta$-Gromov-Witten invariants of $\bp^1$ from topological recursion applied to the spectral curve
$$S^\Theta_{\bp^1}=(C,x,y,B)=\left(\bc,\quad x=z+\frac{1}{z},\quad y=\frac{z}{z^2-1},\quad B=\frac{dzdz'}{(z-z')^2}\right).$$
We first state the result for stationary $\Theta$-Gromov-Witten invariants which consist of insertions of $\omega\in H^2(\bp^1)$.  Consider
\begin{equation}  \label{Dgn} 
D^{\Theta}_{g,n}(x_1,...,x_n)=\sum_{\bf b}\left\langle \Theta\cdot\prod_{i=1}^n \tau_{b_i}(\omega) \right\rangle^g_d\cdot\prod_{i=1}^n(b_i+1)!x_i^{-b_i-2}dx_i.
\end{equation}
\begin{theorem}  \label{TRGWP1th}
For $2g-2+n>0$, $D^{\Theta}_{g,n}(x_1,...,x_n)$ is an analytic expansion at $\{x_i=x(z_i)=\infty\}$ of the correlators $\omega_{g,n}$ of topological recursion applied to the spectral curve $S^\Theta_{\bp^1}$.
\end{theorem}
In the two exceptional cases $(g,n)=(0,1)$ and $(0,2)$, the invariants $\omega_{g,n}$ are not analytic at $x_i=\infty$.  We get analytic expansions around a branch of $\{x_i=\infty\}$ by removing their singularities at $x_i=\infty$ as follows:  
\begin{equation}  \label{excep} 
\omega_{0,1}+\frac{dx_1}{x_1}\sim D^{\Theta}_{0,1}(x_1),\quad\omega_{0,2}-\displaystyle\frac{dx_1dx_2}{(x_1-x_2)^2}\sim D^{\Theta}_{0,2}(x_1,x_2).
\end{equation}
The usual Gromov-Witten invariants of $\bp^1$ satisfy a similar result to Theorem~\ref{TRGWP1th}.  Analogous to the series $D^{\Theta}_{g,n}(x_1,...,x_n)$, define
\[D_{g,n}(x_1,...,x_n)=\sum_{\bf b}\left\langle \prod_{i=1}^n \tau_{b_i}(\omega) \right\rangle^g_d\cdot\prod_{i=1}^n(b_i+1)!x_i^{-b_i-2}dx_i\]
where $d$ is determined by $\sum_{i=1}^n b_i=2g-2+2d$.
\begin{theorem}[\cite{DOSSIde,NScGro}]  \label{TRGWP1}
For $2g-2+n>0$, $D_{g,n}(x_1,...,x_n)$ is an analytic expansion at $\{x_i=\infty\}$ of the correlators $\omega_{g,n}$ of topological recursion applied to the spectral curve 
$$S_{\bp^1}=(C,x,y,B)=\left(\bc,\quad x=z+\frac{1}{z},\quad y=\log{z},\quad B=\frac{dzdz'}{(z-z')^2}\right).$$
\end{theorem}
Note that the function $y(z)=\log{z}$ is locally well-defined up to a constant around $z=\pm 1$, the zeros of $dx$, so the invariants $\omega_{g,n}$ are well-defined.  Theorem~\ref{TRGWP1} was proven for $g=0$ and $g=1$ in \cite{NScGro} and proven in general in \cite{DOSSIde}.   

In \cite{DOSSIde} it was proven that the partition function of the spectral curve $S_{\bp^1}$ coincides with the ancestor partition function:
$$Z^{S_{\bp^1}}(\h,\{\bar{t}_k,\bar{s}_k\})=\hat{Z}_{\bp^1}(\h,\{\bar{t}_k,\bar{s}_k\}).$$
The variables $t_k$ and $s_k$ correspond to the basis $\{1,\omega\}$ of $H^*(\bp^1)$. This requires a linear change of coordinates from the coordinates $u^{\alpha}_k$ in the partition function in Definition~\ref{TRpart}.  
$$Z^{S_{\bp^1}}(\h,\{\bar{t}_k,\bar{s}_k\})=\left.\exp\sum_{g,n}\frac{\h^{2g-2}}{n!}\omega^{S_{\bp^1}}_{g,n}\right|_{\mathfrak{t}_k(z)=\bar{t}_k,\mathfrak{s}_k(z)=\bar{s}_k}$$
where we have replaced the auxiliary differentials defined in \eqref{Vdiff} with $\mathfrak{t}_k(z)= U^1_k(z)-iU^2_k(z)$ and $\mathfrak{s}_k(z)= U^1_k(z)+iU^2_k(z)$ or equivalently
\begin{equation} \label{Tdiff}
\begin{aligned}
\mathfrak{t}_0(z)&=\frac{1}{2}\bigg(\frac{1}{(z - 1)^2} - \frac{1}{(z + 1)^2}\bigg)dz,\quad \mathfrak{t}_{k+1}(z)=d\left(\frac{\mathfrak{t}_k(z)}{dx(z)}\right), \quad k\geq 0\\
\mathfrak{s}_0(z)&=\frac{1}{2}\bigg(\frac{1}{(z - 1)^2} + \frac{1}{(z + 1)^2}\bigg)dz,\quad
\mathfrak{s}_{k+1}(z)=d\left(\frac{\mathfrak{s}_k(z)}{dx(z)}\right), \quad k\geq 0.
\end{aligned}
\end{equation}

We can apply Proposition~\ref{thetadecomp} to produce a new spectral curve with partition function the $\Theta$-Gromov-Witten invariants of $\bp^1$.
\begin{theorem}  \label{TRGW1}
The topological recursion partition function of the spectral curve 
$$S^{\Theta}_{\bp^1}=(C,x,y,B)=\left(\bc,\quad x=z+\frac{1}{z},\quad y=\frac{z}{z^2-1},\quad B=\frac{dzdz'}{(z-z')^2}\right)$$
coincides with the ancestor partition function of $\Theta$-Gromov-Witten invariants of $\bp^1$
$$\hat{Z}^{\Theta}_{\bp^1}(\h,\{\bar{t}_k,\bar{s}_k\})=\left.\exp\sum_{g,n}\frac{\h^{2g-2}}{n!}\omega^{S^{\Theta}_{\bp^1}}_{g,n}\right|_{\mathfrak{t}_k(z)=\bar{t}_k,\mathfrak{s}_k(z)=\bar{s}_k}
$$
where the differentials $\mathfrak{t}_k(z)$ and $\mathfrak{s}_k(z)$ are defined in \eqref{Tdiff}.
\end{theorem}
\begin{proof}
By Proposition~\ref{thetadecomp}, the ancestor partition functions of Gromov-Witten invariants of $\bp^1$ and $\Theta$-Gromov-Witten invariants of $\bp^1$ use the same element $R(z)$ of the twisted loop group, the same topological field theory and differ only by the shifted translation term $T_0(z)$.  The element $R(z)$ is encoded by the $(C,x,B)$ part of the spectral curve hence $(C,x,B)$ is the same for $S^{\Theta}_{\bp^1}$ and $S_{\bp^1}$.  It remains to find $y$ for $S^{\Theta}_{\bp^1}$ which determines the translation term $T_0(z)$.  We will write the spectral curve $S_{\bp^1}=(C,x,Y,B)$ where $Y=z$ so as not to confuse it with $y$.  Choose a local coordinate $s$ around a zero $p_{\alpha}$ of $dx$, by $x=\frac12 s^2+x(p_{\alpha})$.  In this local coordinate, the odd part of $y$ is 
$$y^{\alpha}_{\rm odd}=\sum_{k=0}^{\infty}y^{\alpha}_{2k-1}s^{2k-1}$$
which we have allowed to have a simple pole at $s=0$.  It is proven in \cite{CNoTop} that the translation action on partition functions induced from the translation action \eqref{transl} on CohFTs is realised in topological recursion by $u^{\alpha}_k\mapsto u^{\alpha}_k- (2k-1)!!\frac{y^{\alpha}_{2k-1}}{y^{\alpha}_{-1}}$ for $k>0$ if $y^{\alpha}_{-1}\neq 0$. Equivalently the $\alpha$ component of the (series-valued) vector $T_0(z)$ is
$$(T_0(z))_{\alpha}=-\sum_{k=1}^{\infty} (2k-1)!!\frac{y^{\alpha}_{2k-1}}{y^{\alpha}_{-1}}z^{k}.$$

Locally the shift $T(z)\mapsto T(z)/z=T_0(z)$ is achieved by differentiation $Y\mapsto dY/dx=y$ since near $p_\alpha$ we have $x=\frac12 s^2$, $Y=\sum_{k=0}^{\infty}Y^{\alpha}_{2k+1}s^{2k+1}$, and $dY^{\alpha}_{\rm odd}/dx=\sum_{k=0}^{\infty}(2k+1)Y^{\alpha}_{2k+1}s^{2k-1}$ which maps to 
$$(T_0(z))_{\alpha}=-\sum_{k=0}^{\infty} (2k-1)!!\cdot\frac{(2k+1)Y^{\alpha}_{2k+1}}{Y^{\alpha}_{1}}z^{k+1}=\frac{1}{z}T(z).
$$
Differentiate globally to get %in Theorem~\ref{TRGWP1} with 
$$Y(z)=\frac{dy}{dx}(z)=\frac{y'(z)}{x'(z)}=\frac{1/z}{1-1/z^2}=\frac{z}{z^2-1}.
$$
as required.
\end{proof}
\begin{proof}[Proof of Theorem~\ref{TRGW}] 
By Theorem~\ref{TRGW1} the ancestor $\Theta$-Gromov-Witten invariants of $\bp^1$ are encoded in the correlators $\omega^{S^{\Theta}_{\bp^1}}_{g,n}$.  Hence by Proposition~\ref{th:ancdest} the stable descendant $\Theta$-Gromov-Witten invariants of $\bp^1$ are also encoded in the correlators $\omega^{S^{\Theta}_{\bp^1}}_{g,n}$.  For the unstable terms, the genus zero 2-point function $\langle\tau_{k_1}(e_{\alpha_1})\tau_{k_2}(e_{\alpha_2})\rangle^0_{\bf d}$ of Gromov-Witten invariants and $\Theta$-Gromov-Witten invariants coincide,  and the former is known to be encoded in $\omega^{S_{\bp^1}}_{0,2}=\omega^{S^{\Theta}_{\bp^1}}_{0,2}$ by \eqref{excep}, so the same is true for the latter.  The genus zero 1-point function can be calculated as follows.  By the definition of the genus zero 1-point invariants in \eqref{def01}, 
$$\left\langle\Theta\cdot\tau_k(\omega)\right\rangle^0_{\bp^1}:=-\left\langle\tau_0(1)\tau_k(\omega)\right\rangle^0_{\bp^1}=-\left\langle\tau_{k-1}(\omega)\right\rangle^0_{\bp^1}$$
where the last equality uses the string equation for usual Gromov-Witten invariants.  The degree $d\in\bn$ is uniquely determined by the insertions so it is hidden.  Hence
$$\sum_{d>0}\left\langle \Theta\cdot\tau_{2d-1}(\omega) \right\rangle^0_d \frac{(2d)!}{x^{2d+1}}=-\sum_{d>0}\left\langle\tau_{2d-2}(\omega) \right\rangle^0_d\frac{(2d)!}{x^{2d+1}}=\frac{d}{dx}\sum_{d>0}\left\langle\tau_{2d-2}(\omega) \right\rangle^0_d\frac{(2d-1)!}{x^{2d}}=\frac{df}{dx}
$$
where $f=\log{z}+\log{x}$ for $z=2/(x+\sqrt{x^2-4})$, is proven in \cite{NScGro}.  Hence, for $\omega^0_1=-ydx=-\frac{z}{z^2-1}dx$ we have
$$\sum_{d>0}\left\langle \Theta\cdot\tau_{2d-1}(\omega) \right\rangle^0_d \frac{(2d)!}{x^{2d+1}}dx=\omega^0_1+\frac{dx}{x}
$$
as required.  Note that the $dx/x$ term is there to cancel $-dx/x$ in the expansion of $\omega^0_1$ at $x=\infty$.
\end{proof}
\begin{proof}[Proof of Theorem~\ref{partint}]
The integral in \eqref{partleg} can be equivalently expressed as a matrix integral over $N\times N$ Hermitian matrices with eigenvalues contained in the interval $[-2,2]$
\begin{equation}  \label{legendre}
\int_{H_N[-2,2]}\exp{(NV(M))}DM
\end{equation}
where the measure is Euclidean on the entries of the matrix and $V(M)=\displaystyle\sum_{k\geq 0} s_k\ {\rm tr}(M^{k+1})$.  It is known as the Legendre ensemble.

It is proven in \cite{BGuAsy} that the cumulants of the resolvents of the Legendre ensemble have asymptotic expansions as $N\to\infty$ given by 
$$\left\langle {\rm tr} \left(\frac{1}{x_1-M}\right)\cdots {\rm tr} \left(\frac{1}{x_n-M}\right)\right\rangle^c\stackrel{N\to\infty}{\sim}\sum_{g>0} N^{2-2g-n}W_{g,n}(x_1,...,x_n)$$
which define $W_{g,n}(x_1,...,x_n)$.  These $W_{g,n}(x_1,...,x_n)$ satisfy loop equations which coincides with topological recursion applied to a spectral curve obtained from the large $N$ limit of the resolvent
$$y(x)=\lim_{N\to\infty}N^{-1}\left\langle {\rm tr} \left(\frac{1}{x-M}\right)\right\rangle.
$$
It is well-known that $y(x)$ coincides with the spectral curve $S^{\Theta}_{\bp^1}$, or it can be seen explicitly at the beginning of the proof of Theorem~\ref{qcproof}.    Furthermore, $W_{0,2}(x_1,x_2)$ is well-known to coincide with $\omega_{0,2}-\displaystyle\frac{dx_1dx_2}{(x_1-x_2)^2}$.  Hence $W_{g,n}(x_1,...,x_n)$ satisfies topological recursion for the spectral curve $S^{\Theta}_{\bp^1}$ or more precisely gives an expansion of $\omega^{S^{\Theta}_{\bp^1}}_{g,n}$ at $x=\infty$.  

By Theorem~\ref{TRGW}, the ancestor $\Theta$-Gromov-Witten invariants invariants are encoded in the correlators $\omega^{S^{\Theta}_{\bp^1}}_{g,n}$.  Furthermore, the same basis of differentials \eqref{Tdiff}, whose coefficients give the ancestor Gromov-Witten invariants, is used for $S_{\bp^1}$ and $S^{\Theta}_{\bp^1}$ since the differentials depend only on the data of $(C,x,B)$ which coincides for $S_{\bp^1}$ and $S^{\Theta}_{\bp^1}$.  The linear change of coordinates used in \eqref{ancdes} to relate ancestor and descendant Gromov-Witten invariants is achieved by taking the linear function $\Res_{\infty}x^k\cdot$ on the correlators.  In other words,
$$\left\langle \prod_{i=1}^n \tau_{b_i}(\omega) \right\rangle^g_d=\Res_{x_1=\infty}...\Res_{x_n=\infty}\prod_{i=1}^n\frac{x_i^{b_1+1}}{(b_i+1)!}\omega^{S_{\bp^1}}_{g,n}.
$$
But the linear change of coordinates used in \eqref{ancdest} to relate ancestor and descendant $\Theta$-Gromov-Witten invariants is the same hence we also have: 
$$\left\langle\Theta \cdot\prod_{i=1}^n \tau_{b_i}(\omega) \right\rangle^g_d=\Res_{x_1=\infty}...\Res_{x_n=\infty}\prod_{i=1}^n\frac{x_i^{b_1+1}}{(b_i+1)!}\omega^{S^{\Theta}_{\bp^1}}_{g,n}.
$$
Thus we have proven that the formula $D^\Theta_{g,n}(x_1,...,x_n)$ in \eqref{Dgn} is an expansion of the correlators  $\omega^{S^{\Theta}_{\bp^1}}_{g,n}$, and in particular $D^\Theta_{g,n}(x_1,...,x_n)$ satisfy the same recursion with the same initial conditions, as $W_{g,n}(x_1,...,x_n)$.  Hence $D_{g,n}(x_1,...,x_n)=W_{g,n}(x_1,...,x_n)$ and their respective partition functions $Z_{\bp^1}^{\Theta}(\h,\{s_k\},\{t_k=0\})$ and $\int_{H_N[-2,2]}\exp{(NV(M))}DM$ coincide up to a constant for $\h=1/N$.
\end{proof}

The equivariant Gromov-Witten invariants of $\bp^1$ with respect to the $\bc^*$ action were studied in \cite{OPaEqu}.   The equivariant $\Theta$-Gromov-Witten invariants of $\bp^1$ are well-defined since the group acts only on the target $\bp^1$.   In \cite{FLZEyn} the equivariant Gromov-Witten invariants of $\bp^1$ are shown to arise from topological recursion applied to the spectral curve
$$ x=z+1/z+w\log z,\quad y=\log z
$$
where $w$ is the equivariant parameter of the $\bc^*$ action.  When $w=0$ this gives the spectral curve $S_{\bp^1}$.   It is shown in \cite{FLZEyn} that this is example is semisimple with $R(z)$ and $T(z)$ determined by the spectral curve.  The techniques of this paper should enable the calculation of the equivariant $\Theta$-Gromov-Witten invariants of $\bp^1$ using the spectral curve
$$ x=z+1/z+w\log z,\quad y=\frac{z}{z^2+wz-1}
$$
obtained by replacing $y$ with $dy/dx$ as described in the proof of Theorem~\ref{TRGW1}.

\subsection{Quantum curve}   \label{sec:qc}
A quantum curve of a spectral curve $S=\{(x,y)\in\bc^2\mid P(x,y)=0\}$ is a linear differential equation 
\begin{equation}   \label{waveq}
\widehat{P}(\hat{x}, \hat{y}) \, \psi(p,\h)=0
\end{equation}
where $p\in S$, $\h$ is a formal parameter, and $\widehat{P}(\hat{x}, \hat{y})$ is a differential operator-valued non-commutative quantisation of the plane curve with $\hat{x}=x\cdot$ and $\hat{y}=\h\frac{d}{dx}$, and in particular $[\hat{x},\hat{y}]=-\h$.

The wave function $\psi(p,\h)$ is given by
$$
\psi(p,\h)=\exp(\h^{-1}S_0(p)+S_1(p)+\h S_2(p)+\h^2 S_3(p)+...)
$$
and the $S_k(p)$ are calculated recursively via \eqref{waveq}, which is known as the WKB method.   The $S_k(p)$ are meromorphic functions on $S$ for $k>1$.  The operator $\frac{d}{dx}$ acts on meromorphic functions via composition of the exterior derivative followed by division by $dx$, %---which sends meromorphic differentials on $C$ to meromorphic functions on $C$---
and coincides with usual differentiation of a meromorphic function around a point where $x$ defines a local coordinate.
In many cases, the wave function is related to topological recursion applied to $S$ via
$$
S_k(p)=\sum_{2g-1+n=k}\frac{(-1)^n}{n!}\int^p_{p_0}\int^p_{p_0}...\int^p_{p_0}\omega^S_{g,n}(p_1,...,p_n)
$$
and this relation is proven for $S=S^{\Theta}_{\bp^1}$ in Theorem~\ref{qcproof} below.
See \cite{BEyRec,GSuApo,NorQua} for a description of the relationship of quantum curves to topological recursion.  
\begin{theorem}  \label{qcproof}
The quantum curve of $S^{\Theta}_{\bp^1}$ is given by
$$\left((4-x^2)\h^2\frac{d^2}{dx^2}-2x\h^2\frac{d}{dx}+1+\h\right)Z_{\bp^1}^{\Theta}\Big(s_k=\h\frac{k!}{x^{k+1}},t_k=0\Big)=0.
$$
It is related to topological recursion via
$$Z_{\bp^1}^{\Theta}\Big(s_k=\h\frac{k!}{x^{k+1}},t_k=0\Big)=\exp(\h^{-1}S_0(p)+S_1(p)+\h S_2(p)+\h^2 S_3(p)+...)$$
for $S_k$ defined by
\begin{equation}   \label{wavef}
S_k(p)=\sum_{2g-1+n=k}\frac{(-1)^n}{n!}\int^p_{\infty}\int^p_{\infty}...\int^p_{\infty}\omega^{S^{\Theta}_{\bp^1}}_{g,n}(p_1,...,p_n).
\end{equation}
\end{theorem}
\begin{proof}
Put $\psi(x,\h)=\exp(\h^{-1}S_0(x)+S_1(x)+O(\h))$.  Then 
$$\psi(x,\h)^{-1}\left((4-x^2)\h^2\frac{d^2}{dx^2}-2x\h^2\frac{d}{dx}+1+\h\right)\psi(x,\h)=(4-x^2)\frac{d^2S_0(x)}{dx^2}+1+O(\h).
$$
Let $\h\to 0$, so we see that $\frac{d}{dx}S_0(x)$ defines the spectral curve and indeed $dS_0(x)=-ydx=\omega_{0,1}(p)$ as required.
More generally, we will prove that the wave function agrees with the form in \eqref{wavef} to all orders.   We have a primitive of $(-1)^n\omega^{S^{\Theta}_{\bp^1}}_{g,n}(p_1,...,p_n)=(-1)^nD^{\Theta}_{g,n}(x_1,...,x_n)$  given by
$$F^{\Theta}_{g,n}(x_1,...,x_n)=\sum_{\bf b}\left\langle \Theta\cdot\prod_{i=1}^n \tau_{b_i}(\omega) \right\rangle^g_d\cdot\prod_{i=1}^n(b_i)!x_i^{-b_i-1}$$
which we specialise to $x_i=x$ to get
$$S_k(x)=\sum_{2g-1+n=k}\frac{1}{n!}\sum_{{\bf b},d}\left\langle \Theta\cdot\prod_{i=1}^n \tau_{b_i}(\omega) \right\rangle^g_d\cdot\prod_{i=1}^n(b_i)!x^{-b_i-1}
$$
leading to
$$\psi(x,\h)=Z_{\bp^1}^{\Theta}\Big(s_k=\h\frac{k!}{x^{k+1}},t_k=0\Big)
$$
as required.  

Next, we prove the differential equation.  Define the resolvents $W_n(x_1,...,x_n)$ by
$$W_n(x_1,...,x_n)=\frac{c_N}{N!}\int_{-2}^2\int_{-2}^2...\int_{-2}^2dt_1...dt_N\prod_{i<j}(t_i-t_j)^2\prod_{i=1}^n\sum_{j=1}^N\frac{1}{x_i-t_j}$$
where $1/c_N=Z_{\bp^1}^{\Theta}\Big(\hbar=1/N,s_k=0,t_k=0\Big)$ is a constant which will play no role in the solution of a linear differential equation.  Hence
\begin{align*}
\sum_{n=1}^{\infty}\frac{1}{n!}\int^x_\infty\int^x_\infty...\int^x_\infty W_n(x_1,...,x_n)&=\sum_{n=1}^{\infty}\frac{c_N}{N!}\int_{-2}^2\int_{-2}^2...\int_{-2}^2dt_1...dt_N\prod_{i<j}(t_i-t_j)^2\frac{1}{n!}\left(\log\prod_{j=1}^N(x-t_j)\right)^n\\
&=\frac{c_N}{N!}\int_{-2}^2\int_{-2}^2...\int_{-2}^2dt_1...dt_N\prod_{i<j}(t_i-t_j)^2\prod_{j=1}^N(x-t_j)
\end{align*}
The left hand side is the exponential of the same multiple integral formula specialised at $x$ over its cumulants $W^g_n(x_1,...,x_n)=D^{\Theta}_{g,n}(x_1,...,x_n)/\prod dx_i$.  Thus we have (up to a constant in $x$) 
$$\psi(x,\h)=\frac{1}{N!}\int_{-2}^2\int_{-2}^2...\int_{-2}^2dt_1...dt_N\prod_{i<j}(t_i-t_j)^2\prod_{j=1}^N(x-t_j)=d_NP_N(\frac12 x)$$
where $P_N(x)$ is the degree $N$ Legendre polynomial up to a constant in $x$.  The second equality is proven in \cite{AomJac}.
Hence $\psi(x,\h)$ is an asymptotic expansion of $d_NP_N(\frac12 x)$ as $1/\h=N^2\to\infty$.

The Legendre polynomial satisfies the Legendre differential equation
$$(1-x^2)\frac{d^2}{dx^2}P_N(x)-2x\frac{d}{dx}P_N(x)+N(N+1)P_N(x)=0
$$
hence after rescaling $x\mapsto x/2$ and setting $N=1/\h$ we have
$$(4-x^2)\h^2\frac{d^2}{dx^2}\psi(x,\h)-2x\h^2\frac{d}{dx}\psi(x,\h)+(1+\h)\psi(x,\h)=0
$$
as required.
\end{proof}
\begin{remark}
Note that we can incorporate $\tau_0(1)$ insertions into the quantum curve by defining the wave function:
$$\psi(x,\h,t)=Z_{\bp^1}^{\Theta}\Big(s_k=\h\frac{k!}{x^{k+1}},t_0=\h t,t_k=0,k>0\Big).
$$
which satisfies the differential equation
$$\left((4-x^2)\h^2\frac{d^2}{dx^2}-2x\h^2\frac{d}{dx}+(1-t)(1-t+\h)\right)\psi(x,\h,t)=0.
$$
This corresponds to the spectral curve $S_t=\{(x,y)\in\bc^2\mid (x^2-4)y^2=(1-t)^2\}$ which replaces $y\mapsto(1-t)y$ hence $\omega_{g,n}\mapsto (1-t)^{2-2g-n}\omega_{g,n}$ hence $S_k\mapsto(1-t)^{1-k}S_k$ which corresponds to $\h\mapsto\h/(1-t)$

\end{remark}

%Does $N=1$ give insight into the quantum curve?

Put $\psi(x,\h)=\hat{\psi}(x,\h)x^{1/\h}$ where $\hat{\psi}(x,\h)$ is a series in $x^{-1}$.  The differential equation becomes
$$\left((4-x^2)\h^2\frac{d^2}{dx^2}-2x\h\left(1+\h-\frac{4}{x^2}\right)\frac{d}{dx}+\frac{4}{x^2}(1-\h)\right)\hat{\psi}(x,\h)=0
$$
Decompose $\hat{\psi}$ into its degree $d$ components:
$$\hat{\psi}(x,\h)=\sum_{d=0}^\infty\hat{\psi}_d(\h)x^{-2d}=\hat{\psi}_0(\h)+\hat{\psi}_1(\h)x^{-2}+...$$
The differential equation gives $(6\h^2-4\h(1+\h))\hat{\psi}_1(\h)=4(1-\h)\hat{\psi}_0(\h)$.  Since $\hat{\psi}_0(\h)=1$ we can calculate $\hat{\psi}_1(\h)$:
\begin{equation}   \label{deg1}
\hat{\psi}_1(\h)=\frac{2(1-\h)}{\h(\h-2)}
\end{equation}
which we will need later.

\subsection{Toda equation}  \label{sec:toda}

The Toda equation for a function $Z(t_0)=Z(\h,\{s_k\},t_0)$ is given by
\begin{equation}   \label{toda}
\frac{\partial^2}{\partial s_0^2}\log Z(t_0)=\frac{Z(t_0+\h)Z(t_0-\h)}{Z(t_0)^2}
\end{equation}
where we show the dependence of $Z$ only the variable $t_0$.  It was conjectured by Eguchi-Yang \cite{EYaTop} and proven by Okounkov-Pandharipande \cite{OPaGro} that the partition function $Z_{\bp^1}(\h,\{s_k\},t_0)$ of Gromov-Witten invariants of $\bp^1$ satisfies \eqref{toda}.  In this section we prove that the partition function $Z^{\Theta}_{\bp^1}(\h,\{s_k\},t_0)$ of $\Theta$-Gromov-Witten invariants of $\bp^1$ also satisfies \eqref{toda}.

The free energy of the Gromov-Witten invariants of $\bp^1$ is given by
\begin{align*}
F_{\bp^1}(\h,\{s_k,t_k\})=&\sum_{g\geq 0}  \h^{2g-2}F_g=  \sum_{g,d}\h^{2g-2}\left\langle\exp\left\{\sum_{k\geq 0}^{\infty}\tau_k(\omega)s_k+\tau_k(1)t_k\right\}\right\rangle^g_d+\frac{1}{12}\log\h\\
=&\h^{-2}\left((1-t_1)^2+\frac{\frac{1}{2}t_0^2s_0}{1-t_1}+s_0(1-t_1)+\frac{1}{2}s_0^2+\frac{\frac{1}{6}t_0^3s_1}{(1-t_1)^2}+...\right)+\frac{1}{12}\log\frac{\h}{1-t_1}-\frac{\frac{1}{24}s_0}{1-t_1}+...
\end{align*}
and $Z_{\bp^1}(\h,\{s_k,t_k\})=\exp F_{\bp^1}(\h,\{s_k,t_k\})$.  The dilaton equation
$$ \frac{\partial}{\partial t_1}Z_{\bp^1}(\h,\{s_k,t_k\})=\left(\h\frac{\partial}{\partial\h}+\sum_{k}t_k\frac{\partial}{\partial t_k}+\sum_{k}s_k\frac{\partial}{\partial s_k}\right)Z_{\bp^1}(\h,\{s_k,t_k\})
$$
is equivalent to the statement that $Z_{\bp^1}(\h,\{s_k,t_k\})$ is homogeneous of degree zero in $\h$, $\{s_k\}$, $\{t_k,\ k\neq 1\}$ and $q=t_1-1$.  For example, the $t_1$ terms combine in $F_{\bp^1}$ to give the degree zero terms $-\frac16 t_0^2s_0\h^{-2}q^{-1}$ and $\frac{1}{12}\log(-\h q^{-1})$.  The boundary term $\frac{1}{12}\log\h$ in $F_{\bp^1}$ as discussed in \cite{EHYTop}, is required to counteract the term $\frac{1}{12}t_1$.   The presence of a $\log\h$ term is explained quite generally in Section~\ref{GW}.   Moreover, the Toda equation holds only when the $\frac{1}{12}\log\h$ term is present.   

\begin{theorem}
The partition function satisfies the Toda equation
$$\h^2(1-t_0)^2\frac{\partial^2}{\partial s_0^2}\log Z_{\bp^1}^{\Theta}(\h,\{s_k\},t_0)=\frac{Z_{\bp^1}^{\Theta}(t_0+\h)Z_{\bp^1}^{\Theta}(t_0-\h)}{Z_{\bp^1}^{\Theta}(t_0)^2}.
$$
\end{theorem}
\begin{proof}
The main idea is to use the random matrix integral of Theorem~\ref{partint}.  So we first need to incorporate the insertions of the class $\tau_0(1)$ into the integral.

The dilaton equation in Proposition~\ref{diltheta} restricted to stationary invariants and $\tau_0(1)$ insertions is 
\begin{equation}  \label{dilth}
\left\langle\tau_0(1)\cdot\Theta\cdot\prod_{i=1}^l\tau_{b_i}(1)  \prod_{i=l+1}^n\tau_{b_i}(\omega)\right\rangle=(2g-2+n)\left\langle\Theta\cdot\prod_{i=1}^l\tau_{b_i}(1)  \prod_{i=l+1}^n\tau_{b_i}(\omega)\right\rangle
\end{equation}
which is equivalent to the PDE \eqref{dilPDE} so that $Z^{\Theta}_{\bp^1}(\h,\{s_k\},t_0)$ is homogeneous of degree 0 in $\h$, $q=t_0-1$ and $\{s_k\}$.  Hence
$$Z^{\Theta}_{\bp^1}(\h,\{s_k\},t_0)=Z^{\Theta}_{\bp^1}\left(\frac{\h}{1-t_0},\left\{\frac{s_k}{1-t_0}\right\}\right)
$$
where any variable not appearing is set to zero.  Thus in the integral of Theorem~\ref{partint}, we set $N=\frac{1-t_0}{\h}$ and replace $s_k$ by $\frac{s_k}{1-t_0}$.

Now we follow a rather standard argument showing that the random matrix integral satisfies the Toda equation.  Here we write $Z_N=\frac{1}{c}Z^{\Theta}_{\bp^1}(1/N,\{s_k\},t_0)$ where $c$ is the constant in \eqref{partleg} which is unimportant since the following Toda equation is invariant under rescaling $Z(N)$ by a constant.
\begin{equation}  \label{TodaM}
\frac{\partial^2}{\partial s_0^2}\log Z(N)=\frac{Z(N+1)Z(N-1)}{Z(N)^2}
\end{equation}
Define $h_j=h_j(s_0,s_1,...)$ for $j\in\bn$ to be the norms of a set of monic orthogonal polynomials
$$\int^2_{-2}\exp\left(\sum s_jx^j\right)p_j(x,\{s_i\})p_k(x,\{s_i\})dx=\langle p_j,p_k\rangle=h_j^2\delta_{j,k}.
$$
(i) A standard proof found in, for example, \cite{ForLog} equation (5.74) or  \cite{MehMet} equation (3.10), shows
$$Z(N,\{s_k\})=\prod_{j=1}^N h_{j-1}^2.$$
(ii) In \cite{ForLog} equation (5.122) it is proven that
$$\frac{\partial}{\partial s_0}\left(\frac{1}{h_k^2}\frac{\partial}{\partial s_0}h_k^2\right)=\left(\frac{h_{k+1}}{h_k}\right)^2+\left(\frac{h_{k-1}}{h_{k-2}}\right)^2-2\left(\frac{h_{k}}{h_{k-1}}\right)^2,\quad k>1
$$
with boundary relations $\frac{\partial}{\partial s_0}\left(\frac{1}{h_1^2}\frac{\partial}{\partial s_0}h_1^2\right)=\left(\frac{h_{2}}{h_1}\right)^2-\left(\frac{h_{1}}{h_{0}}\right)^2$  and $\frac{\partial}{\partial s_0}\left(\frac{1}{h_0^2}\frac{\partial}{\partial s_0}h_0^2\right)=\left(\frac{h_{1}}{h_0}\right)^2$. \\ 
\\
(iii) It follows immediately that
$$\sum_{i=1}^N\frac{\partial}{\partial s_0}\left(\frac{1}{h_{i-1}^2}\frac{\partial}{\partial s_0}h_{i-1}^2\right)=\frac{h_N^2}{h_{N-1}^2}=\sum_{i=1}^N\left(\left(\frac{h_i}{h_{i-1}}\right)^2-\left(\frac{h_{i-1}}{h_{i-2}}\right)^2\right)
$$
which is equivalent to \eqref{TodaM}.

\vspace{.5cm}
Finally notice that $N+1=\frac{1-t_0}{\h}+1=\frac{1-t_0+\h}{\h}$ so $N\mapsto N+1$ is equivalent to $t_0\mapsto t_0-\h$.  Similarly, $N\mapsto N-1$ is equivalent to $t_0\mapsto t_0+\h$.   Since $s_0$ has been replaced by $\frac{s_0}{1-t_0}$ the second derivative is rescaled by $(1-t_0)^2$ and \eqref{TodaM} becomes
$$\h^2(1-t_0)^2\frac{\partial^2}{\partial s_0^2}\log Z_{\bp^1}^{\Theta}(\h,\{s_k\},t_0)=\frac{Z_{\bp^1}^{\Theta}(t_0+\h)Z_{\bp^1}^{\Theta}(t_0-\h)}{Z_{\bp^1}^{\Theta}(t_0)^2}
$$
as required.
\end{proof}
In Theorem~\ref{Todath} we use the variable $\tilde{s}_0=s_0/(1-t_0)$ in order to get the same Toda equation for both Gromov-Witten invariants of $\bp^1$ and $\Theta$-Gromov-Witten invariants of $\bp^1$.  Getzler \cite{GetTod} proved that the Toda equation for the full partition function $Z_{\bp^1}$ of the usual Gromov-Witten invariants which includes all $t_k$ follows from the Toda equation for $Z$ with $t_k=0$ for $k>0$ using Virasoro relations.  Perhaps an analogous result holds for $Z^{\Theta}_{\bp^1}$.

The Toda equation determines the partition function uniquely from its degree zero part.  In the Toda equation, the derivative $\frac{\partial^2}{\partial s_0^2}$ reduces the degree by 1---it describes the removal of $\tau_0(\omega)^2$---and multiplication by $(1-t_0)^2$ does not change the degree.  Hence, with respect to the decomposition
$$Z_{\bp^1}^{\Theta}(\h,\{s_k\},t_0)=\sum_{d\geq 0}Z_d(\h,\{s_k\},t_0)
$$
the degree $d$ part $Z_d$ is determined by $Z_{d'}$ for $d'<d$.

\subsection{The 1-point series}   \label{1-point}
In this section we calculate the 1-point stationary invariants $\langle\Theta\cdot\tau_{2d-1}(\omega)\rangle^g$. 
It is proven in \cite{GKRMat} that 
$$
W_1(x)=\sum_g\h^{2g-2}W^g_1(x)=\sum_{g,d}\h^{2g-2}\left\langle \Theta\cdot\tau_{2d-1}(\omega) \right\rangle^g_d\cdot\frac{(2d)!}{x^{2d+1}}.
$$
satisfies a differential equation
\begin{equation}  \label{1ptDE}
\frac{d^3W_1}{dx^3}+\frac{8x}{x^2-4}\frac{d^2W_1}{dx^2}+\frac{14x^2-24}{(x^2-4)^2}\frac{dW_1}{dx}+\frac{4x}{(x^2-4)^2}W_1=4\h^{-2}\left(\frac{1}{x^2-4}\frac{dW_1}{dx}+\frac{x}{(x^2-4)^2}W_1\right).
\end{equation}
Equivalently \eqref{1ptDE} can be written 
$$DW_1^g=EW_1^{g-1},\qquad D=\frac{4}{xy^2}\frac{d}{dx}+4,\quad E=\frac{1}{xy^4}\frac{d^3}{dx^3}+\frac{8}{y^2}\frac{d^2}{dx^2}+\frac{14x^2-24}{x}\frac{d}{dx}+4$$
for $y=1/\sqrt{x^2-4}$.  With respect to the basis of monomials $\{y^k,k\in\bz\}$, the differential operator $D$ is diagonal, and the differential operator $E$ is upper triangular.  We have:
$$  Dy^k=4(1-k)y^k,\qquad Ey^k=(1-k)(k-2)^2y^k-4k^2(k-2)y^{k+2}.
$$
From this we deduce that the series $W^g_1(x)$ is a degree $2g+1$ polynomial in $y$.  A Taylor expansion of $y=y(x)$ around $x=\infty$ will retrieve the series in $x$.  The first few polynomials are:
\begin{align*}
W^0_1&=-y\\
W^1_1&=\frac12y^3\\
W^2_1&=\frac18y^3+\frac98y^5\\
W^3_1&=\frac{1}{32}y^3+\frac{45}{16}y^5+\frac{225}{16}y^7\\
W^4_1&=\frac{1}{128}y^3+\frac{819}{128}y^5+\frac{15750}{128}y^7+\frac{55125}{128}y^9\\
W^5_1&=\frac{1}{512}y^3+\frac{1845}{128}y^5+\frac{108675}{128}y^7 +\frac{1157625}{128}y^9+\frac{6251175}{256}y^{11}
\end{align*}
The series $W^g_1(x)$ collects the 1-point invariants of a given genus.  We can also collect the 1-point invariants of a given degree $d$. Define 
$$G_d=\sum_{g}\h^{2g-2}\left\langle \Theta\cdot\tau_{2d-1}(\omega) \right\rangle^g_d,\quad d>0
$$
so that $\displaystyle W_1(x)=\sum_d G_d\frac{(2d)!}{x^{2d+1}}$.
The differential equation \eqref{1ptDE} implies the following recursion for $G_d$ which uniquely determines $G_d$ from the initial conditions $G_{-1}=0$ and $G_0=-\h^{-2}$.
$$d^2(1-(d-\tfrac12)^2\h^{2})G_d=(1-(2d^2-3d+\tfrac32)\h^{2})G_{d-1}+\h^2G_{d-2}
$$
The recursion implies that $G_d$ is a rational function of $\h^2$ with simple poles at $\h^2=0$ and $\h^2=4/(2m-1)^2$ for $m=1,...,d$.  The low degree series are as follows.
\begin{align*}
G_0&=-\h^{-2}\\
G_1&=-\h^{-2}+\frac{1}{4-\h^2}\\
G_2&=\frac{1}{4}\left(-\h^{-2}+\frac{13}{8}\frac{1}{4-\h^2}+\frac{3}{8}\frac{1}{4-9\h^2}\right)\\
G_3&=\frac{1}{36}\left(-\h^{-2}+\frac{135}{64}\frac{1}{4-\h^2}+\frac{99}{128}\frac{1}{4-9\h^2}+\frac{15}{128}\frac{1}{4-25\h^2}\right)\\
G_4&=\frac{1}{4!^2}\left(-\h^{-2}+\frac{2577}{1024}\frac{1}{4-\h^2}+\frac{1179}{1024}\frac{1}{4-9\h^2}+\frac{305}{1024}\frac{1}{4-25\h^2}+\frac{35}{1024}\frac{1}{4-49\h^2}\right)
\end{align*}
It is shown in \cite{GKRMat} that $W_1(x)$ is the large $N$ ($=1/\h$) asymptotic expansion of a sequence of Taylor series in $x^{-1}$.
$$W_1(x)=N(P_{N-1}(x)Q_N'(x)-P_{N}(x)Q_{N-1}'(x))
$$
where $P_N(x)$ is the $N$th Legendre polynomial and $Q_N(x)$ the $N$th Legendre function of the second kind.  The latter is analytic at $x=\infty$ and satisfies $Q_N'(x)=O(x^{-N-2})$, hence $P_{N-1}(x)Q_N'(x)$ and $P_{N}(x)Q_{N-1}'(x)$ are both Taylor series in $x^{-1}$.  The poles of $G_d$ corresponding to $N=n-\frac12$, a half-integer, arise because $Q_N(x)$ diverges as $N$ tends to a half-integer.

The behaviour of the 1-point series $G_d$ is rather different than the behaviour of the 1-point series for the usual Gromov-Witten invariants of $\bp^1$.  It is proven in \cite{OPaGro} that 
$$ H_d=\sum_{g}\h^{2g-2}\left\langle\tau_{2g-2+2d}(\omega) \right\rangle^g_d=\frac{\h^{-2}}{d!^2}S(\h)^{2d-1},\quad S(\h)=\frac{\sinh{\frac{\h}{2}}}{\h/2}.
$$
The methods of \cite{OPaGro} use a Fock space representation of the Gromov-Witten invariants of $\bp^1$ which produces the 1-point series $H_d$.  The difference in behaviour of the 1-point series $G_d$ and $H_d$ suggests that a Fock space approach to $\Theta$-Gromov-Witten invariants of $\bp^1$ might not exist.  It was shown in \cite{PanTod} that the formula for $H_d$ is a consequence of the Toda equation and the divisor equation.  Although the $\Theta$-Gromov-Witten invariants of $\bp^1$ satisfy the Toda equation, they satisfy a different type of divisor equation.

\subsection{Hodge classes and the 0-point series}  \label{0-point}

Degree zero stable maps to a target variety $X$ consist of a stable curve and its image in $X$, hence
$${\modm}_{g,n}(X,0)\cong\overline{\modm}_{g,n}\times X.
$$
The virtual class is strictly smaller than the moduli space when $\dim X>0$.  It is given by the top Chern class of a rank $g\times\dim X$ bundle
$$[{\modm}_{g,n}(X,0)]^{vir}=c_{\text{top}}(\Lambda^\vee\boxtimes TX)$$
where $\Lambda\to\overline{\modm}_{g,n}$ is the rank $g$ Hodge bundle with fibres over smooth curves given by the vector space of holomorphic differentials and $\Lambda^\vee$ is its dual.  When $X=\bp^1$, since $\boxtimes$ is the tensor product of the pull-back bundles we have $c(\Lambda^*\boxtimes TX)=\pi_1^*c(\Lambda^*)\pi_2^*c(T\bp^1)=\pi_1^*(1-\lambda_1+\lambda_2+...+(-1)^g\lambda_g)\pi_2^*(1+2\omega)$ hence the top Chern class is $c_g(\Lambda^*\boxtimes T\bp^1)=(-1)^g(\pi_1^*\lambda_g-2\pi_1^*\lambda_{g-1}\pi_2^*\omega)$ and the push-forward is given by
$$ p_*[{\modm}_{g,n}(\bp^1,0)]^{vir}=(-1)^{g-1}2\lambda_{g-1}\in H^*(\overline{\modm}_{g,n}).
$$
We can use knowledge of the $\Theta$-Gromov-Witten invariants of $\bp^1$ to learn more about the structure of the classes $\Theta_{g,n}$ and in particular their integral next to the Hodge class.  We have  
$$ \int_{\overline{\modm}_{g}}\lambda_{g-1}\Theta_{g}=(-1)^{g-1}\frac12F_g
$$
This is analogous to calculations of the Hodge integrals $ \int_{\overline{\modm}_{g,n}}\lambda_{g-1}\prod_i\psi_i^{m_i}
$ using Gromov-Witten invariants of $\bp^1$ carried out in \cite{FPaHod,GPaVir}.

The 0-point series  
$$ \sum\h^{2g-2}F_g:=\sum\h^{2g-2}\left\langle \Theta \right\rangle ^{g}_0=\sum\h^{2g-2}\int_{[{\modm}_{g}(\bp^1,0)]^{vir}} \Theta^{\bp^1}_{g}
$$
can be determined by the Toda equation.

Consider the degree 1 stationary $\Theta$-Gromov-Witten invariants of $\bp^1$.  We can calculate 
$$H_1(\h):=\sum\langle\Theta\cdot\tau_0(\omega)\tau_0(\omega)\rangle^g_1\h^{2g-2}
$$
from the quantum curve together with
$$G_1(\h)=\sum_{g}\h^{2g-2}\left\langle \Theta\cdot\tau_{1}(\omega) \right\rangle^g_1
=-\h^{-2}+\frac{1}{4-\h^2}$$
as follows.  The degree 1 component of the wave function calculated in \eqref{deg1} is
$$\hat{\psi}_1(\h)=\frac{2(1-\h)}{\h(\h-2)}=\h G_1(\h)+\frac12\h^2H_1(\h).
$$
hence
$$H_1(\h)=\h^{-2}+\frac{1}{4-\h^2}.
$$
The Toda equation gives
$$\frac{Z_0(\frac{\h}{1-\h})Z_0(\frac{\h}{1+\h})}{Z_0(\h)^2}
=\sum_{g}\langle\Theta\cdot\tau_0(\omega)\tau_0(\omega)\rangle^g_1\h^{2g}=\h^2H_1(\h)=\frac{4}{4-\h^2}.
$$
Hence
$$\log Z_0=\frac14\log\h-\frac{1}{64}\h^2+\frac{1}{256}\h^4-\frac{17}{6144}\h^6+\frac{31}{8192}\h^8-\frac{691}{81920}\h^{10}+\frac{5461}{196608}\h^{12}+...
$$
and it is not difficult to show the exact expression
$$ F_g=\frac{(1-2^{-2g})}{g(g-1)}B_{2g}.
$$
A generating series looks nicest when we include the $g=1$ term which can be achieved by considering $\langle\tau_0(1)\rangle^g=(2g-2)F_g$.
$$
\sum_{g>0} \h^{2g-1}\frac{\langle\tau_0(1)\rangle^g}{(2g-1)!}=\tanh\frac{\h}{4}.
$$

%$$ (2g-2)F_g=(-1)^{g-1}(2g-1)!(1-2^{-2g})2^{3-2g}\frac{\zeta(2g)}{\pi^{2g}}.$$
%$$\sum F_g\frac{t^{2g-2}}{(2g-3)!}=-\frac14\tanh^2\frac{t}{4}.$$
%$$\int_0^{\infty}e^{-\frac{x}{\h}}\tanh\frac{x}{4}dx=\h^3\frac{d}{d\h}\log Z_0.$$ 
Thus we have 
$$\sum_{g>1}\h^{2g-2}\int_{\overline{\modm}_{g}}\lambda_{g-1}\Theta_{g}=\frac{1}{128}\h^2+\frac{1}{512}\h^4+\frac{17}{12288}\h^6+\frac{31}{16384}\h^8+\frac{691}{163840}\h^{10}+...
$$
and $\int_{\overline{\modm}_{1,1}}\lambda_0\Theta_{1,1}=\frac18$.

For $n>0$, the degree 0 $\Theta$-Gromov-Witten invariants involve only $\tau_0(1)$ insertions since
$$\deg\left\{\prod_{i=1}^l\tau_{b_i}(1)  \prod_{i=l+1}^n\tau_{b_i}(\omega)\right\}=0=\dim \left\{[\overline{\modm}_{g,n}(\bp^1,0)]^{\text{vir}}\cap \Theta_{g,n}^{\bp^1}\right\}\quad\Leftrightarrow\quad l=n \text{\ and\  }b_i=0.$$ 
They are determined by the 0-point invariants via the dilaton equation \eqref{dilth} hence
$$Z_0=\exp\left\{\sum\left(\frac{\h}{1-t_0}\right)^{2g-2}F_g\right\}.$$
Likewise, the Hodge integrals are determined from the integrals $ \int_{\overline{\modm}_{g}}\lambda_{g-1}\Theta_{g}$ for $g>1$ and $\int_{\overline{\modm}_{1,1}}\Theta_{1,1}$ via
$ \int_{\overline{\modm}_{g,n+1}}\lambda_{g-1}\Theta_{g,n+1}=(2g-2+n)\int_{\overline{\modm}_{g,n}}\lambda_{g-1}\Theta_{g,n}$.

\end{document}